\documentclass[a4paper]{article}

\usepackage{lipsum}
\usepackage{amsfonts}
\usepackage{amssymb}
\usepackage{graphicx}
\usepackage{epstopdf}
\usepackage{algorithmic}

\usepackage{amsmath,amsfonts,amssymb}
\usepackage{float}

\ifpdf
  \DeclareGraphicsExtensions{.eps,.pdf,.png,.jpg}
\else
  \DeclareGraphicsExtensions{.eps}
\fi

\newcommand{\D}{\displaystyle}
\newcommand{\dpartial}[2]{\displaystyle\frac{\partial #1}{\partial #2}}
\newcommand{\veps}{\varepsilon}
\newcommand{\dsum}{\D\sum}
\newcommand{\dint}{\D\int}

\newcommand{\mub}{\overline{\mu}}

\newcommand{\uk}{\textbf{u}}

\newcommand{\vk}{\textbf{v}}
\newcommand{\wk}{\textbf{w}}
\newcommand{\zk}{\textbf{z}}
\newcommand{\unmu}{U_N(\mu)}

\newcommand{\fk}{\textbf{f}}
\newcommand{\gk}{\textbf{g}}
\newcommand{\intO}[1]{\dint_\Omega #1\; d\Omega}
\newcommand{\intK}[1]{\dsum_{K\in \mathcal{T}_h}\dint_K #1\;d\Omega}
\newcommand{\dO}{d\Omega}
\newcommand{\cD}{\mathcal{D}}
\newcommand{\R}{\mathbb{R}}
\newcommand{\Tnmu}{T_N^\mu}

\newcommand{\TN}[1]{T_N(#1;\mu)}

\newcommand{\ra}{\rightarrow}
\newcommand{\cDb}{\cD_{train}}

\newcommand{\romu}{\rho_T}
\newcommand{\normlt}[1]{\|#1\|_{0,3,\Omega}}
\newcommand{\normlc}[1]{\|#1\|_{0,4,\Omega}}
\newcommand{\normld}[1]{\|#1\|_{0,2,\Omega}}
\newcommand{\normh}[1]{\|#1\|_{1,2,\Omega}}
\newcommand{\normT}[1]{\|#1\|_T}
\newcommand{\normX}[1]{\|#1\|_X}

\newcommand{\difu}{\uk^1-\uk^2}
\newcommand{\difuh}{\uk^1_h-\uk^2_h}

\newcommand{\difZ}{Z^1_h-Z^2_h}

\newcommand{\difUh}{U^1_h-U^2_h}

\newcommand{\cR}[1]{\mathcal{R}(#1;\mu)}
\newcommand{\DA}[1]{\mathcal{DA}(#1;\mu)}
\newcommand{\DAn}{\DA{\unmu}}
\newcommand{\Hw}[1]{H(#1;\mu)}
\newcommand{\difH}{\Hw{Z_h^1}-\Hw{Z_h^2}}
\newcommand{\en}{\epsilon_N(\mu)}
\newcommand{\taun}{\tau_N(\mu)}
\newcommand{\dn}{\Delta_N(\mu)}
\newcommand{\bmu}{\beta_N(\mu)}
\newcommand{\dual}[1]{\left<#1,V_h\right>}
\newcommand{\ximu}{\xi(\mu)}
\newcommand{\nk}{\textbf{n}}
\newcommand{\xip}{\xi^p}
\newcommand{\zetav}{\zeta^\textbf{v}}

\newtheorem{lemma}{Lemma}
\newtheorem{proposition}{Proposition}
\newtheorem{theorem}{Theorem}
\newtheorem{remark}{Remark}

\newenvironment{proof}{\textbf{Proof.}}{\par \begin{flushright}
$\square$
\end{flushright}}
\begin{document}

\title{On a certified Smagorinsky reduced basis turbulence model}

\author{
Tom\'{a}s Chac\'{o}n Rebollo\thanks{IMUS \& Departamento de Ecuaciones Diferenciales y An\'{a}lisis Num\'{e}rico 
, Apdo. de correos 1160, Universidad de Sevilla, 41080 Seville, Spain. chacon@us.es, edelgado1@us.es}  
\and Enrique Delgado \'{A}vila\footnotemark[1]
\and Macarena G\'{o}mez M\'{a}rmol\thanks{Departamento de Ecuaciones Diferenciales y An\'{a}lisis Num\'{e}rico, Apdo. de correos 1160, Universidad de Sevilla, 41080 Seville, Spain. macarena@us.es}
\and Francesco Ballarin\thanks{mathLab, Mathematics Area, SISSA, International School for Advanced Studies, via Bonomea 265, I-34136 Trieste, Italy. francesco.ballarin@sissa.it, gianluigi.rozza@sissa.it}
\and Gianluigi Rozza\footnotemark[3]
}

\maketitle

\begin{abstract}
In this work we present a reduced basis Smagorinsky turbulence model for steady flows. We approximate the non-linear eddy diffusion term using the Empirical Interpolation Method (\emph{cf.} \cite{EIM1, EIM2}), and the velocity-pressure unknowns by an independent reduced-basis procedure. 
This model is based upon an \emph{a posteriori} error estimation for Smagorinsky turbulence model. The theoretical development of the \emph{a posteriori} error estimation is based on \cite{Deparis} and \cite{Manzoni}, according to the Brezzi-Rappaz-Raviart stability theory, and adapted for the non-linear eddy diffusion term. 
We present some numerical tests, programmed in FreeFem++ (\textit{cf.} \cite{freefem++}), in which we show an speedup on the computation by factor larger than 1000  in benchmark 2D flows.
\end{abstract}



\section{Introduction}
Reduced Order Modeling (ROM) has been successfully used in several fields to provide large reduction in computation cost for the solution of Partial Differential Equations \cite{hestaven, holmes, Patera2002, LibroBR, Stokes1, Rozza2008}.  In fluid mechanics a popular strategy is to use POD to extract the  dominant structures for high-Reynolds flow,  which are then used in a Galerkin approximation of the underlying equations \cite{holmes, sirovich}. There have been a number of recent works combining POD/RB with Variational Multi-Scale models \cite{VMS-Iliescu}, ensemble models \cite{Ensemble1}, flow regularization models \cite{Regularization1, Ballarin} as well as bifurcation problems \cite{Bifurcation4,Bifurcation2,Bifurcation1,Bifurcation3} all in the framework of the incompressible Navier-Stokes equations. Application of the POD-Galerkin strategy to turbulent fluid flows remains a challenging area of research. By construction, ROMs generated using only the first most energetic POD basis functions are not endowed with the dissipative mechanisms associated to the creation of lower size, and less energetic, turbulent scales. Increasing the number of modes creates very large POD-Galerkin ROMs that are still very computationally expensive to solve (\textit{cf.} \cite{hadowell} and references therein). 
\par
A developing way of research to overcome this difficulty is to adapt the standard turbulent closure techniques based upon eddy dissipation to model the effect of the un-resolved ROM modes on the resolved ones. This is based upon the analysis of \cite{couplet}, that shows that the transfer of energy among the POD modes is similar to the transfer of energy among Fourier modes, in fact there is a net energy transfer from low index POD modes to higher index POD modes (\textit{cf.} \cite{iliescu} and references therein).

In this paper we address an alternative strategy, that consists of constructing ROMs of turbulence models, rather than using ROM to construct turbulence models. We assume that the flow under consideration (or, rather, its large scales) is well modeled by the starting turbulence model, at least up to the accuracy required by the targeted application. Our purpose is to construct fast solvers for the turbulence model which is a highly non-linear mathematical system of equations (often with larger and more complex non-linearities than the Navier-Stokes equations) and needs large computational times to be solved. Current engineering applications for design, optimization and control require repeated queries to turbulent models, so there is a heavy interest in following this approach.

We address the systematic construction of a Reduced Basis (RB) Smagorinsky turbulence model (\textit{cf.} \cite{Smago}), which is the basic Large Eddy Simulation (LES) turbulence model, in which the effect of the subgrid scales on the resolved scales is modeled by eddy diffusion terms (\textit{cf.} \cite{TomasSmago, Sagaut}). It is an intrinsically discrete model, since the eddy viscosity term depends on the mesh size. 

RB methods for incompressible fluid flows were first introduced for the Stokes equations (see e.g. \cite{Stokes2, Stokes3, Stokes4}). The reduced basis is constructed by means of greedy algorithms, an \textit{a posteriori} error bound estimator based upon the dual norm of the residual and the inf-sup constant developed for the Stokes problem. This idea was then extended for the Navier-Stokes problem (see e.g. \cite{Deparis,Dparis-Rozza,Manzoni,Patera}), developing the \textit{a posteriori} error estimator taking into account the Brezzi-Rappaz-Raviart (BRR) theory (\textit{c.f.} \cite{BRR}). We extend in this paper the approach for the reduction of Navier-Stokes equations to the Smagorinsky turbulence model.
     
Unlike the Navier-Stokes problem that has an affine formulation with respect to the physical parameters, the Smagorinsky's eddy viscosity term leads to a non-affine formulation with respect to the parameter for the Smagorinsky model. We use the Empirical Interpolation Method (EIM), introduced in \cite{EIM1} and \cite{EIM2}, to approximate the non-linear eddy viscosity term of the Smagorinsky model, obtaining an affine formulation with respect to the parameter. Thanks to this technique, we can store in the offline phase parameter-independent matrices, obtaining in the online phase a fast and highly accurate computation of the eddy diffusion term.

The construction of the reduced spaces for velocity and pressure is made by means of a greedy algorithm. To reduce the time to compute the errors within this algorithm, we have developed an \textit{a posteriori} error bound estimator based on the Brezzi-Rappaz-Raviart theory (see \cite{BRR}). The use of a norm that takes into account the eddy viscosity effects allows a better initialization of the Greedy algorithm.

We have performed several tests of the reduced model to solve 2D step and cavity flows, with Reynolds number ranging in intervals in which a steady solution is known to exist. We obtain speed-up rates of several order of magnitude, where errors are normalized with respect to the finite element solution below $10^{-4}$.

The structure of this paper is as follows. In section \ref{sec::FE}, we present the continuous and discrete problems we work with. The reduced basis method is considered in \ref{sec::RBM}, where we describe the Greedy algorithm that we use to select the different snapshots. After that, in section \ref{sec::numerical}, we present the numerical analysis that we need in order to assure the well-possedness, based on the BRR theory, of the discrete problem presented in section \ref{sec::FE}. The construction and analysis of the \textit{a posteriori} error bound estimator is presented in section \ref{sec::post}. Then, in section \ref{sec::EIM}, we explain more in detail how we treat the non-linear eddy viscosity corresponding with the Smagorinsky term. Finally, we present some numerical results in section \ref{sec::results}, where we show the reduction of the computational time in two different tests.

\section{Finite Element Smagorinsky model}\label{sec::FE}
To formulate the Smagorinsky turbulence model, let $\Omega$ be a bounded polyhedral domain in $\R^d, (d=2,3)$. We assume that its boundary is split into $\Gamma=\Gamma_D\cup\Gamma_N$, where $\Gamma_D=\Gamma_{D_g}\cup\Gamma_{D_0}$ is the boundary relative to the non-homogeneous and homogeneous Dirichlet boundary conditions, and $\Gamma_N$ to the Neumann conditions. 

Let $\{\mathcal{T}_h\}_{h>0}$ a family of affine-equivalent and conforming triangulations of $\overline{\Omega}$, formed by triangles or quadrilaterals ($d=2$), tetrahedra or hexaedra ($d=3$). As usual the parameter $h$ is the maximum diameter $h_K$ among the elements $K\in\mathcal{T}_h$. 

Although the Smagorinsky model is intrinsically discrete, it can be interpreted as a discretization of a continuous model. We next present this model to clarify its relationship with the Navier-Stokes equations. In this way, the ``continuous'' Smagorinsky turbulence model is formulated as

\begin{equation}\label{NS}\left\{\begin{array}{ll}
\wk\cdot\nabla\wk+\nabla p-\nabla\cdot\left(\left(\dfrac{1}{\mu}+\nu_T(\wk)\right)\nabla\wk\right)=\fk&\mbox{ in }\Omega\vspace{0.1cm}\\

\nabla\cdot\wk=0&\mbox{ in } \Omega\vspace{0.1cm}\\
\wk=\gk_D&\mbox{ on }\Gamma_{D_g}\vspace{0.1cm}\\
\wk=0&\mbox{ on }\Gamma_{D_0}\vspace{0.1cm}\\

-p\nk+\left(\dfrac{1}{\mu}+\nu_T(\wk)\right)\dpartial{\wk}{\textbf{n}}=0&\mbox{ on }\Gamma_N
\end{array}\right.
\end{equation}
 where $\mu$ is the Reynolds number, $\wk=\wk(\mu)$ is the velocity field and $p=p(\mu)$ is the pressure, both depending on the Reynolds number. 

The eddy diffusion term is given by $\nu_T(\wk)=C_S^2\D\sum_{K\in\mathcal{T}_h}h_K^2\big|\nabla\wk_{|_K}\big|\chi_K$, where $\big|\cdot\big|$ denotes the Frobenius norm in $\R^{d\times d}$, and $C_S$ is the Smagorinsky constant \cite{phdSamu}.

Let us consider the spaces $Y=\{\vk\in H^1(\Omega):\vk|_{\Gamma_D}=0\}$ for velocity and $M=L^2(\Omega)$ for pressure. We assume that there exists a lift function $\uk_D\in(H^1(\Omega))^d$, such that $\uk_D|_{\Gamma_{D_g}}=\gk_D$, $\uk_D|_{\Gamma_{D_0}}=0$, and $\nabla\cdot\uk_D=0$ in $\Omega$.
With those conditions, we assure that the lifted velocity $\uk=\wk-\uk_D$ is still incompressible and satisfies the homogeneous Dirichlet boundary conditions on $\Gamma_D$.  We will assume that $\fk\in (L^2(\Omega))^d$, and $\gk_D\in(H^{1/2}(\Omega))^d$.

Let $Y_h\subset Y$ and $M_h\subset M$ be two finite subspaces of $Y$ and $M$. We  consider the following variational discretization of problem (\ref{NS}), actually, the ``true'' Smagorinsky model:

\begin{equation}\label{FV}\left\{\begin{array}{l}
\mbox{Find } (\uk_h,p_h)=(\uk_h(\mu),p_h(\mu))\in Y_h\times M_h\mbox{ such that } \forall\vk_h\in Y_h, \forall q_h\in M_h \vspace{0.2cm}\\

\begin{array}{ll}
a(\uk_h,\vk_h;\mu)+b(\vk_h,p_h;\mu)+ a_S(\wk_h;\wk_h,\vk_h;\mu) \\
+c(\uk_h,\uk_h,\vk_h;\mu)+c(\uk_D,\uk_h,\vk_h;\mu)+c(\uk_h,\uk_D,\vk_h;\mu)=F(\vk_h;\mu)\vspace{0.1cm}\\
b(\uk_h,q_h;\mu)=0\end{array}\end{array}\right.
\end{equation}
where $\wk_h=\uk_h+\uk_D$; the bilinear forms $a(\cdot,\cdot;\mu)$ and $b(\cdot,\cdot;\mu)$ are defined by
\[\]\[
a(\uk,\vk;\mu)=\frac{1}{\mu}\int_\Omega\nabla\uk:\nabla\vk\,d\Omega,\qquad b(\vk,q;\mu)=-\int_\Omega(\nabla\cdot\vk)q\,d\Omega;
\]
the trilinear form, $c(\cdot,\cdot,\cdot;\mu)$, and the non-linear Smagorinsky term, $a_S(\cdot;\cdot,\cdot;\mu)$, are given by
\[
c(\zk,\uk,\vk;\mu)=\int_\Omega(\zk\cdot\nabla\uk)\vk\,d\Omega,\qquad a_S(\zk;\uk,\vk;\mu)=\int_\Omega\nu_T(\zk)\nabla\uk:\nabla\vk\,d\Omega.
\]

Finally, the linear form $F(\cdot;\mu)$ is defined by
\[
F(\vk;\mu)=\left<\fk,\vk\right>-a(\uk_D,\vk;\mu)-c(\uk_D,\uk_D,\vk;\mu),
\]
where $\left<\cdot,\cdot\right>$ stands for the duality pairing between $Y'$ and $Y$, $Y'$ being the dual space of $Y$.

The solution of problem $(\ref{FV})$ is intended to approximate the large-scales component of the solution of the Navier-Stokes problem (i.e., problem $(\ref{NS})$ with $\nu_T=0)$.

Let us define the norms relative to the spaces $Y$ and $M$. For the velocity space $Y$, we consider a weighted inner product, $(\cdot,\cdot)_T$, defined as
\begin{equation}\label{normT}
(\uk,\vk)_T=	\intO{\left[\frac{1}{\mub}+\nu_T^*\right]\nabla\uk:\nabla\vk}\qquad\forall\uk,\vk\in Y,
\end{equation} 
where $\nu_T^*=\nu_T(\wk(\mub))$, $\mub=\arg\D\min_{\mu\in\cD}\left\{\D\sum_{K\in\mathcal{T}_h}(C_Sh_K)^2\D\min_{x\in K}|\nabla\wk(\mu)|(x)\chi_K(x)\right\}$, and $\wk(\mu)$ is the velocity solution of (\ref{NS}).
This inner product induces a norm linked to the eddy diffusion term, $\|\cdot\|_T=(\cdot,\cdot)_T^{1/2}$. As the functions of $Y$ vanish on $\Gamma_D$, then, this norm is equivalent to the usual $H^1$ norm. This norm will turn out to be crucial to apply our error estimator in the RB construction by the Greedy algorithm. For the pressure space $M$, we will use the usual $L^2$-norm, denoted by $\normld{\cdot}$.
  
For the sake of simplicity of notation, let us denote by $X$ the product space $X=Y\times M$ and, by extension, $X_h=Y_h\times M_h\subset X$. We also define the $X$-norm as
\begin{equation}\label{Xnorm}
\|U\|_X=\sqrt{\|\uk\|_T^2+\normld{p_u}^2}\qquad \forall U=(\uk,p_u)\in X,
\end{equation}
With this notation, we can rewrite the variational problem (\ref{FV}) as:

\begin{equation}\label{pb}\left\{\begin{array}{l}
\mbox{Find }U_h(\mu)\in X_h \mbox{ such that}\vspace{0.2cm}\\
A(U_h(\mu),V_h;\mu)=F(V_h;\mu)\qquad\forall V_h\in X_h.\end{array}\right.
\end{equation}

In this formulation, the operator $A$ is given by,
\begin{equation}\label{opA}
A(U_h,V_h;\mu)=\frac{1}{\mu}A_0(U_h,V_h)+A_1(U_h,V_h)+A_2(U_h;V_h)+A_3(U_h;V_h),
\end{equation}
where we denote $V_h=(\vk_h,p_v)$, and
\[
\begin{array}{rll}
A_0(U,V)&\!\!\!\!=\intO{\nabla\uk:\nabla\vk}, \quad A_2(U;V)=\intO{(\uk\cdot\nabla\uk)\vk},\\                                                
A_1(U,V)&\!\!\!\!=\intO{[(\nabla\cdot\uk)p_v-(\nabla\cdot\vk)p_u]}
+\intO{(\uk_D\cdot\nabla\uk)\vk}
+\intO{(\uk\cdot\nabla\uk_D)\vk}, \\
A_3(U;V)&\!\!\!\!=\intO{\nu_T(\uk+\uk_D)\,\nabla(\uk+\uk_D):\nabla\vk}.
\end{array}
\]

\section{Reduced Basis Problem}\label{sec::RBM}
In this section, we present the reduced basis method for the Smagorinsky turbulence model. We will focus on the Greedy algorithm. This is an adaptation of the RB method for Navies-Stokes equations (\textit{c.f.} \cite{Deparis, Manzoni}). We assume that the Reynolds number ranges on a compact interval $\cD\subset\R$.  For the startup of the Greedy algorithm, we choose an arbitrary parameter value $\mu^1\in\cDb$, and we compute the corresponding first snapshot $(\uk_h(\mu^1),p_h(\mu^1))$, solution of the finite element problem (\ref{pb}). Here $\cDb\subset\cD$ is a discrete set where we select the possible values of $\mu$. We will denote by $N_{\max}$ the maximum number of basis functions. 

In order to guarantee the inf-sup stability of the RB approximation \cite{Ballarin,Stokes4,Stokes3}, let us consider the so-called inner pressure \emph{supremizer} operator $T_p^\mu:M_h\rightarrow Y_h$, as
\begin{equation}\label{sup}
\left(T_p^\mu q_h,\vk_h\right)_T=b(q_h,\vk_h;\mu)\quad\forall\vk_h\in Y_h.
\end{equation}

In this way, we define our first reduced velocity and pressure spaces as
\begin{equation}
M_1=\mbox{span}\{\xi_1^p:=p_h(\mu^1)\},\quad
Y_1=\mbox{span}\{\zeta_k^{\vk}:=\uk_h(\mu^1), T_p^\mu\xi_1^p\}.
\end{equation}

To add a new element to the reduced space, choose the (N+1)-th value of $\mu\in\cDb$ as
\begin{equation}\label{Greedyerr}
\mu^{N+1}=\arg\max_{\mu\in\cDb}\normX{U_h(\mu)-U_N(\mu)}, \quad 1\le N\le N_{\max};
\end{equation}
where $U_N(\mu)$ is the solution of the discrete model on the current reduced-basis space $X_N$:

\begin{equation}\label{rb}
\left\{\begin{array}{l}
\mbox{Given }\mu\in\cD,\mbox{ find }U_N(\mu)\in X_N \mbox{ such that}\vspace{0.2cm}\\
A(U_N(\mu),V_N;\mu)=F(V_N;\mu)\qquad\forall V_N\in X_N.\end{array}\right.
\end{equation}

Since the computation of $\normX{U_h(\mu)-U_N(\mu)}$ may be very expensive due to the computation of the finite element solution $U_h(\mu)$ for all $\mu\in \cDb$, we consider an inexpensive \textit{a posteriori} error estimator $\Delta_N$, constructed in Section \ref{sec::post}, and we define
\begin{equation}
\mu^{N+1}=\arg\max_{\mu\in\cDb}\Delta_N(\mu), \quad 1\le N\le N_{\max}.
\end{equation}

Once we compute the optimum $\mu^{N+1}$, we add to the reduced space the new snapshots $(\uk_h(\mu^{N+1}),p_h(\mu^{N+1}))$, solution of the finite problem (\ref{pb}). We also have to add to the velocity space the \textit{supremizer} corresponding to the pressure snapshot. Thus, the new reduced space, is $X_{N+1}=Y_{N+1}\times M_{N+1}$, where
\begin{equation}\label{MN}
M_{N+1}=\mbox{span}\{\xip_k:=p_h(\mu^k)\}_{k=1}^{N+1}, \quad
Y_{N+1}=\mbox{span}\{\zetav_k:=\uk_h(\mu^k),~ T_p^\mu\xip_k\}_{k=1}^{N+1}.
\end{equation}
   
Note that the construction of these spaces is hierarchical, i.e., $X_1\subset X_2\subset\dots\subset X_{N_{\max}}\subset X_h$. Finally, in order to avoid ill-conditioned matrices in the solution of $(\ref{rb})$, we orthonormalize the reduced velocity space $Y_N$ with respect the norm $\normT{\cdot}$, and the reduced pressure space $M_N$ with respect the $L^2$-norm.

We summarize the Greedy algorithm:
\begin{enumerate}
\item Set $\mu^1$, and compute $\uk_h(\mu^1), p_h(\mu^1)$ and $T_p^\mu p_h(\mu^1)$, and the reduced spaces $Y_1, M_1$. 
\item For $k\ge2$, compute $\Delta_{k-1}(\mu), \forall\mu\in\cD_{train}$ and set $\mu^{k}=\arg\D\max_{\mu\in\cD_{train}}\Delta_{k-1}(\mu).$
\item Compute $\uk_h(\mu^{k}), p_h(\mu^k)$ and $T_p^\mu p_h(\mu^k)$, and then the reduced spaces $Y_k, M_k$.
\item Stop if $\D\max_{\mu\in\cD_{train}}\Delta_k(\mu)<\varepsilon_{RB}$.
\end{enumerate}

\section{Well-posedness analysis}\label{sec::numerical}
\hspace{14pt} The well-posedness of the Smagorinsky problem is provided in \cite{TomasSmago} by the classical Brezzi theory (\textit{c.f.} \cite{Brezzi}). The boundedness of the FE solution is provided by this analysis. However, in this section, we analyse the well-posedness of the Smagorinsky FE solution using the more general Brezzi-Rappaz-Raviart (BRR) theory (see e.g. \cite{BRR}). The finality of using the BRR theory instead the Brezzi theory is the construction of the error estimator provided by the BRR for the reduced basis problem. Let us denote the directional derivative,  at $U_h\in X_h$, in the direction $Z_h=(\zk,p_z)\in X_h$, as
$\partial_1A(U_h,\cdot;\mu)(Z_h)$.
If we derive each operator term in (\ref{opA}), we obtain
\[
\begin{array}{ll}
\partial_1A_0(U,V)(Z)&=A_0(Z,V),\quad
\partial_1A_1(U,V)(Z)=A_1(Z,V),\vspace{0.1cm}\\
\partial_1A_2(U;V)(Z)&=\intO{(\uk\cdot\nabla\zk)\vk}+\intO{(\zk\cdot\nabla\uk)\vk},\vspace{0.1cm}\\
\partial_1A_3(U;V)(Z)&=\intO{\nu_T(\uk+\uk_D)\,\nabla\zk:\nabla\vk} \\
				\,	&+\dsum_{K\in\mathcal{T}}\dint_K{(C_Sh_K)^2\dfrac{\nabla(\uk+\uk_D):\nabla\zk}{|\nabla(\uk+\uk_D)|}\big(\nabla(\uk+\uk_D):\nabla\vk\big)\;\dO},\end{array}
\]


For the well posedness of the problem, we have to guarantee the uniform coerciveness and the boundedness of $\partial_1A$ in the sense that for any solution $U_h(\mu)$ of (\ref{pb}), there exist $\beta_0>0$ and $\gamma_0\in\mathbb{R}$ such that $\forall\mu\in\mathcal{D}$,
\begin{equation}\label{conditions}
\begin{array}{l}
0<\beta_0<\beta_h(\mu)\equiv\D\inf_{Z_h\in X_h}\D\sup_{V_h\in X_h}\dfrac{\partial_1A(U_h(\mu),V_h;\mu)(Z_h)}{\|Z_h\|_X\|V_h\|_X},\vspace{0.2cm}\\
\infty>\gamma_0>\gamma_h(\mu)\equiv\D\sup_{Z_h\in X_h}\D\sup_{V_h\in X_h}\dfrac{\partial_1A(U_h(\mu),V_h;\mu)(Z_h)}{\|Z_h\|_X\|V_h\|_X}.\end{array}
\end{equation}

Then, according to the BRR theory (\textit{cf.} \cite{BRR,Rappaz}), it will follow that in a neighbourhood of $U_h(\mu)$ the solution of (\ref{pb}) is unique and bounded in $\normX{\cdot}$ in terms of the data. We will prove this in Section \ref{sec::post}, and as consequence we shall construct the \textit{a posteriori} error bound estimator. 

Since $H^1(\Omega)$ is embedded in $L^4(\Omega)$, let us denote by $C_T$ the Sobolev embedding constant such that 
$\normlc{\vk}\le C_T\normT{\vk}$, for all $\vk\in Y$. Also, let us denote by $C_{\mub}$ the constant such that $\normT{\vk}\le C_{\mub}\normld{\nabla \vk}$, for all $\vk \in Y$. These constant will be used in the following propositions. By standard arguments, it follows

\begin{proposition}\label{prop::cont}
There exists $\gamma_0\in\R$ such that $\forall \mu\in\cD$
\[
|\partial_1A(U_h(\mu),V_h;\mu)(Z_h)|\le \gamma_0\|Z_h\|_X\|V_h\|_X\quad \forall Z_h,V_h\in X_h.
\]
\end{proposition}

Proving that the $\beta_h(\mu)$ inf-sup condition in (\ref{conditions}) is satisfied, we assure that we are in a smooth branch of solutions for the Smagorinsky problem.

\begin{proposition}\label{prop::infsup}
Let $C^\star=C_T^2(C_{\mub}+1)$. Suppose that $\normld{\nabla\uk_D}<\dfrac{1}{C^\star}$, and $\normld{\nabla\uk_h}\le \dfrac{1}{C^\star}-\normld{\nabla\uk_D}$. Then, there exists $\tilde{\beta_h}>0$ such that,
\begin{equation}
\partial_1 A(U_h,V_h;\mu)(V_h)\ge\tilde{\beta_h}\normT{\vk_h}^2\qquad\forall V_h\in X_h.
\end{equation}
\end{proposition}
\begin{proof} We consider $Z_h=V_h$ in $\partial_1 A(U_h,V_h;\mu)(Z_h)$, having
\begin{equation}\begin{array}{ll}
\partial_1A(U_h,V_h;\mu)(V_h)&=\dfrac{1}{\mu}\partial_1A_0(U_h,V_h)(V_h)+\partial_1A_1(U_h,V_h)(V_h)\vspace{0.1cm}\\
&+\;\partial_1A_2(U_h,V_h)(V_h)+\partial_1A_3(U_h;V_h)(V_h).
\end{array}
\end{equation}

As,
\begin{equation}\label{infsup1}
\begin{array}{ll}
\dfrac{1}{\mu}\partial_1A_0(U_h,V_h)(V_h)+\partial_1A_3(W_h;V_h)(V_h)&\!\!\!\!=\intO{\left(\dfrac{1}{\mu}+\nu_T(\wk_h)\right)|\nabla\vk_h|^2}\\
&\hspace{-1cm}+\dsum_{K\in\mathcal{T}}\dint_K{(C_Sh_K)^2\dfrac{|\nabla\wk_h:\nabla\wk_h|^2}{|\nabla\wk_h|}\;\dO},
\end{array}
\end{equation}
\[
\partial_1A_1(U_h,V_h)(V_h)+\partial_1A_2(U_h;V_h)(V_h)=\intO{(\wk_h\cdot\nabla\vk_h)\vk_h}+\intO{(\vk_h\cdot\nabla\wk_h)\vk_h}
\]
\begin{equation}\label{infsup2}
\le C_T^2\big(\normT{\wk_h}+\normld{\nabla\wk_h}\big)\normT{\vk_h}^2\le C_T^2(C_{\mub}+1)\normld{\nabla\wk_h}\normT{\vk_h}^2.
\end{equation}

Since $\dsum_{K\in\mathcal{T}}\dint_K{(C_Sh_K)^2\dfrac{|\nabla\wk_h:\nabla\vk_h|^2}{|\nabla\wk_h|}\;\dO}\ge0$, we have, thanks to (\ref{infsup1}) and (\ref{infsup2}),
\begin{equation}\label{infsup3}
\begin{array}{l}
\partial_1A(U_h,V_h;\mu)(V_h)\ge \intO{\left(\dfrac{1}{\mu}+\nu_T(\wk_h)\right)|\nabla\vk_h|^2}\medskip\\
-C_T^2(C_{\mub}+1)\normld{\nabla\wk_h}\normT{\vk_h}^2\ge (1-C_T^2(C_{\mub}+1)\normld{\nabla\wk_h})\normT{\vk_h}^2\medskip\\
\ge\big(1-C_T^2(C_{\mub}+1)\normld{\nabla\uk_D}
-C_T^2(C_{\mub}+1)\normld{\nabla\uk_h}\big)\normT{\vk_h}^2.
\end{array}
\end{equation}

Thus, if $\normld{\nabla\uk_h}\le\dfrac{1}{C^\star}-\normld{\uk_D}$ and $\normld{\uk_D}\le\dfrac{1}{C^\star} $, there exists $\beta_h>0$ such that,
\[
\partial_1A(U_h,V_h;\mu)(V_h)\ge\tilde{\beta_h}\normT{\vk_h}^2,\qquad \forall V_h\in X_h.
\]

\end{proof}

\begin{remark}
Since the operator $b(\vk_h,q_h;\mu)$ satisfies the discrete inf-sup condition $\alpha\normld{q_h}\le\D\sup_{\vk_h\in Y_h}\dfrac{b(\vk_h,p_h;\mu)}{\normh{\vk}}$, and thanks to Proposition \ref{prop::infsup}, we can prove that the operator $\partial_1 A$ satisfies the inf-sup condition in (\ref{conditions}). See \cite{TechTomas} for more details.
\end{remark}

Observe that as $\|\gk_D\|_{1/2,\Gamma_D}\le\|\wk_h\|_{1,\Omega}\le C_{\Omega}\normld{\nabla \wk_h}$, the condition needed in proposition \ref{prop::infsup}, $\normld{\nabla\wk_h}\le\dfrac{1}{C_T^2(C_{\mub}+1)}$, will only be possible if $\|\gk_D\|_{1/2,\Gamma_D}\le\dfrac{C_\Omega}{C_T^2(C_{\mub}+1)}$; thus the Dirichlet boundary data should be sufficiently small.

\section{\textit{A posteriori} error estimator}\label{sec::post}
\hspace{14pt}
In this section we construct the \textit{a posteriori} error bound estimator for the Greedy algorithm, which selects the snapshots for the reduced space $X_N$. In order to obtain this \textit{a posteriori} error bound estimator, we will take into account the well-posedness analysis of the reduced problem (\ref{pb}) done in the previous section. We start by proving that the directional derivative of the operator $A(\cdot,\cdot;\mu)$ is locally lipschitz.

\begin{lemma}\label{LemmaRho}
There exists a positive constant $\rho_T$ such that, $\forall U_h^1,U_h^2,Z_h,V_h \in X_h$,
\begin{equation}\label{ro}
\left|\partial_1A(U_h^1,V_h;\mu)(Z_h)-\partial_1A(U_h^2,V_h;\mu)(Z_h)\right|\le\rho_T\normX{U_h^1-U_h^2}\normX{Z_h}\|V_h\|_X.
\end{equation}
\end{lemma}

\begin{proof} We have that
\begin{align*}
&\partial_1A(U_h^1,V_h;\mu)(Z_h)-\partial_1A(U_h^2,V_h;\mu)(Z_h)=\intO{((\uk_h^1-\uk_h^2)\cdot\nabla \zk_h)\vk_h} \\
&+\intO{(\zk_h\cdot\nabla(\uk_h^1-\uk_h^2))\vk_h} + \intO{\left(\nu_T(\uk_h^1)-\nu_T(\uk_h^2)\right)\nabla\zk_h:\nabla\vk_h}\\
&+\intK{(C_Sh_K)^2\dfrac{\nabla\uk_h^1:\nabla\zk_h}{|\nabla\uk_h^1|}(\nabla\uk_h^1:\nabla\vk_h)}\\
&-\intK{(C_Sh_K)^2\dfrac{\nabla\uk_h^2:\nabla\zk_h}{|\nabla\uk_h^2|}(\nabla\uk_h^2:\nabla\vk_h)}.
\end{align*}

So, thanks to the triangle inequality, it follows
\begin{equation}\label{derSmago}
\begin{array}{lc}
&\left|\partial_1A(U^1_h,V_h;\mu)(Z_h)-\partial_1A(U^2_h,V_h;\mu)(Z_h)\right|\le\left|\intO{((\uk_h^1-\uk_h^2)\cdot\nabla \zk_h)\vk_h}\right|\\
&+\,\left|\intO{(\zk_h\cdot\nabla(\uk_h^1-\uk_h^2))\vk_h}\right|\\
&+\left|\intK{(C_Sh_K)^2\left(|\nabla\uk_h^1|-|\nabla\uk_h^2|\right)\nabla\zk_h:\nabla\vk_h}\right|\\
&+\left|\intK{(C_Sh_K)^2\dfrac{\nabla\uk_h^1:\nabla\zk_h}{|\nabla\uk_h^1|}(\nabla\uk_h^1:\nabla\vk_h)}\right.\\
&\left.-\intK{(C_Sh_K)^2\dfrac{\nabla\uk_h^2:\nabla\zk_h}{|\nabla\uk_h^2|}(\nabla\uk^2_h:\nabla\vk_h)}\right|.
\end{array}
\end{equation}

We bound each term separately in (\ref{derSmago}). For the first two terms, we use the relation between $\normlc{\cdot}$ and $\normT{\cdot}$ used in Proposition \ref{prop::cont}, and the fact that the T-norm is equivalent to the $H^1$-seminorm.

\[
\left|\intO{((\uk_h^1-\uk_h^2)\cdot\nabla \zk_h)\vk_h}\right|\le\intO{|\uk_h^1-\uk_h^2||\nabla \zk_h||\vk_h||}
\]\[
\le\normlc{\uk_h^1-\uk_h^2}\normld{\nabla\zk_h}\normlc{\vk_h}
\le C_T\normT{\uk^1_h-\uk_h^2}\normT{\zk_h}\normT{\vk_h}
\]\[
\le C_T\|\difUh\|_X\|Z_h\|_X\|V_h\|_X\\,
\]\[
\left|\intO{(\zk_h\cdot\nabla(\uk^1_h-\uk_h^2))\vk_h}\right|\le\intO{|\zk_h||\nabla(\difuh)||\vk_h|}
\]\[
\le\normlc{\zk_h}\normld{\nabla(\difuh)}\normlc{\vk_h}\le C_T\normT{\zk_h}\normT{\difu_h}\normT{\vk_h}
\]\[
\le C_T\|Z_h\|_X\|\difUh\|_X\|V_h\|_X.\\
\]

To bound the third term in (\ref{derSmago}), we use the local inverse inequalities (\textit{cf.}\cite{desinv}),
\[
\left|\intK{(C_Sh_K)^2\left(|\nabla\uk_h^1|-|\nabla\uk_h^2|\right)\nabla\zk_h:\nabla\vk_h}\right|
\]\[
\le\intK{(C_Sh_K)^2\left||\nabla\uk_h^1|-|\nabla \uk^2_h|\right||\nabla\zk_h||\nabla\vk_h|}
\]\[
\le(C_Sh)^2\intO{|\nabla(\difuh)||\nabla\zk_h||\nabla\vk_h|}
\]\[
\le (C_Sh)^2\normlt{\nabla(\difuh)}\normlt{\nabla\zk_h}\normlt{\nabla\vk_h}
\]\[
\le C_S^2 h^{2-d/2}C\normld{\nabla(\difuh)}\normld{\nabla\zk_h}\normld{\nabla\vk_h}
\]\[
\le C_S^2 h^{2-d/2}C\normX{\difuh}\normX{\zk_h}\normX{\vk_h}.
\]

The last term in (\ref{derSmago}) is bounded as follows:
\[
\left|\intK{(C_Sh_K)^2\dfrac{\nabla\uk^1_h:\nabla\zk_h}{|\nabla\uk_h^1|}(\nabla\uk_h^1:\nabla\vk_h)}\right.
\]\[
\left.-\intK{(C_Sh_K)^2\dfrac{\nabla\uk_h^2:\nabla\zk_h}{|\nabla\uk_h^2|}(\nabla\uk_h^2:\nabla\vk_h)}\right|
\]\[
= \left|\intK{(C_Sh_K)^2\left[\frac{\nabla\uk_h^1:\nabla\zk_h}{|\nabla\uk_h^1|}\left(\nabla(\difuh):\nabla\vk_h\right)\right.\right.
\]\[
\left.+\frac{\nabla(\difuh):\nabla\zk_h}{|\nabla\uk_h^2|}(\nabla\uk_h^2:\nabla\vk_h)\right]}
\]\[
\left.+\intK{(C_Sh_K)^2\frac{(|\nabla\uk_h^2|-|\nabla\uk^1_h|)\nabla\uk_h^1:\nabla\zk_h}{|\nabla\uk_h^1||\nabla\uk^2_h|}(\nabla\uk_h^2:\nabla\vk_h)}\right|
\]\[
\le\intK{(C_Sh_K)^2|\nabla\zk_h||\nabla(\difuh)||\nabla\vk_h|}
\]\[
+\intK{(C_Sh_K)^2|\nabla(\difuh)||\nabla\zk_h||\nabla\vk_h|}
\]\[
+\intK{(C_Sh_K)^2\left||\nabla\uk_h^1|-|\nabla\uk^2_h|\right||\nabla\zk_h||\nabla\vk_h|}
\]\[
\le3(C_Sh)^2\normlt{\nabla(\difuh)}\normlt{\nabla\zk_h}\normlt{\nabla\vk_h}
\]\[
\le3C_S^2h^{2-d/2}C\normld{\nabla(\difuh)}\normld{\nabla\zk_h}\normld{\nabla\vk_h}\]\[
\le3C_S^2h^{2-d/2}C\normX{\difUh}\normX{Z_h}\normX{V_h}.
\]

Thus, we have just proved that $
|\partial_1A(U_h^1,V_h;\mu)(Z_h)-\partial_1A(U_h^2,V_h;\mu)(Z_h)|\le\rho_T\normX{\difUh}\normX{Z_h}\|V_h\|_X$, where, $\rho_T=2C_T+4C_Sh^{2-d/2}C.$
\end{proof}

We introduce the following \textit{supremizer} operator $T_N:X_h\ra X_h$, defined as
\begin{equation}\label{TN}
(T_NZ_h,V_h)_X=\partial_1A(U_N(\mu),V_h;\mu)(Z_h)\quad \forall V_h,Z_h\in X_h,
\end{equation}
such that
\begin{equation}
T_NZ_h=\arg\sup_{V_h\in X_h}\dfrac{\partial_1A(U_N(\mu),V_h;\mu)(Z_h)}{\|V_h\|_X}.
\end{equation}

Taking this definition into account, in order to guarantee the well-posedness of the reduced basis problem (\ref{rb}), in the same way as in the finite element problem (\ref{pb}), we define the inf-sup and continuity constants:

\begin{equation}\label{betaN}
0<\beta_{N}(\mu)\equiv\inf_{Z_h\in X_h}\sup_{V_h\in X_h}\dfrac{\partial_1A(\unmu,V_h;\mu)(Z_h)}{\|Z_h\|_X\|V_h\|_X}=\inf_{Z_h\in X_h}\dfrac{\|T_N Z_h\|_X}{\|Z_h\|_X},
\end{equation}
\begin{equation}\label{gammaN}
\infty>\gamma_{N}(\mu)\equiv\sup_{Z_h\in X_h}\sup_{V_h\in X_h}\dfrac{\partial_1A(\unmu,V_h:\mu)(V_h)}{\|Z_h\|_X\|V_h\|_X}=\sup_{Z_h\in X_h}\dfrac{\|T_N Z_h\|_{X}}{\|Z_h\|_X}
\end{equation}

\begin{theorem}\label{teor::unicidad}
Let $\mu\in\cD$, and assume that $\beta_N(\mu)>0$. If problem (\ref{pb}) admits a solution $U_h(\mu)$ such that
\[
\normX{U_h(\mu)-U_N(\mu)}\le\frac{\beta_N(\mu)}{\rho_T},
\]
then this solution is unique in the ball $B_X\left(U_N(\mu),\dfrac{\beta_N(\mu)}{\rho_T}\right)$.
\end{theorem}

\begin{proof} The proof of this theorem is an extension of the proof of Lemma 3.1 in \cite{Deparis}. We  define the following operators:
\begin{itemize}
\item $\cR{\cdot}:X_h\ra X'_h$, defined as 
\begin{equation}\label{cR}
\left<\cR{Z_h},V_h\right>=A(Z_h,V_h;\mu)-F(V_h;\mu), \quad\forall Z_h,V_h\in X_h
\end{equation}
\item $\DA{U_h(\mu)} :X_h\ra X'_h$, defined, for $U_h(\mu)\in X_h$, as 
\begin{equation}\label{Thmu}
\left<\DA{U_h(\mu)}Z_h,V_h\right>=\partial_1A(U_h(\mu),V_h;\mu)(Z_h),\quad\forall Z_h,V_h \in X_h
\end{equation}
\item $H: X_h\ra X_h$, defined as
\begin{equation}\label{H}
H(Z_h;\mu)=Z_h-\DAn^{-1}\cR{Z_h}, \quad\forall Z_h\in X_h
\end{equation}
\end{itemize}

Note that $\DAn$ is invertible thanks to the assumption $\beta_N(\mu)>0$. Also note that $\DAn=T_N$ in $X_h'$. We express
\begin{equation}\label{difH}
\difH=(\difZ)-\DAn^{-1}(\cR{Z_h^1}-\cR{Z_h^2}).
\end{equation}

It holds
\begin{equation}\label{difRes}
\cR{Z_h^1}-\cR{Z_h^2}=\DA{\xi}(\difZ),
\end{equation}
where $\xi=\lambda Z_h^1-(1-\lambda)Z_h^2,$ for some $\lambda\in(0,1)$. To prove this, we define the operator $T:[0,1]\ra\R$, by $T(t)=\left<\cR{tZ_h^1+(1-t)Z_h^2},V_h\right>$, for all $V_h\in X$.
Then, $T(0)=\left<\cR{Z_h^2},V_h\right>$ and $T(1)=\left<\cR{Z_h^1},V_h\right>$. The operator $T$ is differentiable in $(0,1)$ and continuous in $[0,1]$, and
\[
T'(t)=\left<\DA{tZ^2+(1-t)Z^1}(\difZ),V_h\right>.
\]

Thus, (\ref{difRes}) follows from the Mean Value Theorem in $\R$. Now, multiplying (\ref{difH}) by $\DAn$ and applying this last property, we can write
\[
\DAn(\difH)=\left[\DAn-\DA{\xi}\right](\difZ).
\]

Then, thanks to (\ref{ro}) and this last equality, it follows
\[
\left<\DAn(\difH),V_h\right>\le\rho_T\normX{\unmu-\xi}\normX{\difZ}\|V_h\|_X.
\]

Now, applying the definitions of $\beta_N(\mu)$, $T_N$, $\DAn$, and this last property, we can obtain
\[
\beta_N(\mu)\normX{\difH}\|T_N(\difH)\|_X
\]\[
\le\|T_N(\difH)\|_X^2
\]\[
=\big(\TN{\difH},\TN{\difH}\big)_X
\]\[
=\left<\DAn(\difH), \TN{\difH}\right>
\]\[
\le\rho_T\normX{\unmu-\xi}\normX{\difZ}\|\TN{\difH}\|_X
\]

We have proved that $
\normX{\difH}\le\dfrac{\rho_T}{\beta_N(\mu)}\normX{\unmu-\xi}\normX{\difZ}$.

If $Z^1$ and $Z^2$ are in $B_X(\unmu,\alpha)$ then, $\normX{\unmu-\xi}\le\alpha$, and,
\[
\normX{\difH}\le\frac{\romu}{\beta_N(\mu)}\alpha\normX{\difZ}.
\]

Then, $\Hw{\cdot}$ is a contraction if $\alpha<\dfrac{\beta_N(\mu)}{\rho_T}.$ So it follows that there can exist at most one fixed point of $\Hw{\cdot}$ inside $B_X\left(\unmu,\dfrac{\beta_N(\mu)}{\rho_T}\right)$, and hence, at most one solution $U_h(\mu)$ to (\ref{pb}) in this ball.
\end{proof}

At this point, let us define the \textit{a posteriori} error bound estimator by
\begin{equation}\label{delta}
\dn=\frac{\beta_N(\mu)}{2\rho_T}\left[1-\sqrt{1-\taun}\right],
\end{equation}
where $\taun$ is given by:
\begin{equation}
\taun=\frac{4\en\rho_T}{\beta_N^2(\mu)},\label{tau}
\end{equation}
with,
\begin{equation}
\en=\|\cR{\unmu}\|_{X'}.\label{eps}
\end{equation}

The suitability of this \textit{a posteriori} error bound estimator is stated by:

\begin{theorem}\label{Teorprinc}
Assume that $\bmu>0$ and $\taun\le1$ for all $\mu\in\cD$. Then there exists a unique solution $U_h(\mu)$ of (\ref{pb}) such that the error with respect $U_N(\mu)$, solution of (\ref{rb}), is bounded by the \textit{a posteriori} error bound estimator, i.e.,
\begin{equation}\label{err}
\normX{U_h(\mu)-\unmu}\le\dn,
\end{equation}
with effectivity
\begin{equation}\label{efec}
\dn\le\left[\frac{2\gamma_{N}(\mu)}{\beta_N(\mu)}+\taun\right]\normX{U_h(\mu)-\unmu}.
\end{equation}
\end{theorem}

\begin{proof} 
To prove (\ref{err}), let $\alpha>0$ and $Z_h\in X_h$ such that $\normX{\unmu-Z_h}\le\alpha$. We use the notations introduced in the proof of Theorem \ref{teor::unicidad}. We consider
\[
\Hw{Z_h}-\unmu=Z_h-\unmu-\DAn^{-1}\cR{Z_h}
\]\[
=Z_h-\unmu-\DAn^{-1}\left[\cR{Z_h}-\cR{\unmu}\right]
\]\[
-\DAn^{-1}\cR{\unmu}
\]

Multiplying by $\DAn$, we obtain
\[
\left<\DAn(\Hw{Z_h}-\unmu),V_h\right>=\left<\DAn(Z_h-\unmu),V_h\right>
\]\[
-\left<\cR{Z_h}-\cR{\unmu},V_h\right>-\left<\cR{\unmu},V_h\right>,\quad\forall V_h\in X_h.
\]

As in the proof of Theorem \ref{teor::unicidad}, it holds
\[
\cR{Z_h}-\cR{\unmu}=\DA{\xi(\mu)}(Z_h-\unmu),
\]
where $\xi(\mu)=t^*Z_h+(1-t^*)\unmu$, $t^*\in(0,1)$.

Due to this and Lemma \ref{LemmaRho}, we obtain:
\[
\left<\DAn(\Hw{Z_h}-\unmu),V_h\right>=\dual{\DAn(Z_h-\unmu)}
\]\[
-\dual{\DA{\ximu}(Z_h-\unmu)}-\dual{\cR{\unmu}}
\]\[
=\dual{\big(\DAn-\DA{\ximu}\big)(Z_h-\unmu)}-\dual{\cR{\unmu}}
\]\[
\le\rho_T\normX{\unmu-\ximu}\normX{Z_h-\unmu}\|V_h\|_X+\en\|V_h\|_X
\]\[
\le\big(\rho_T\normX{Z_h-\unmu}^2+\en\big)\normX{V_h}
\]

Then, using the same arguments as in Theorem \ref{teor::unicidad},
\[
\beta_N(\mu)\normX{\Hw{Z_h}-\unmu}\normX{\Tnmu(\Hw{Z_h}-\unmu)}\le\|\Tnmu(\Hw{Z_h}-\unmu)\|_{X}^2
\]\[
=\Big(\Tnmu\big(\Hw{Z_h}-\unmu\big),\,\Tnmu\big(\Hw{Z_h}-\unmu\big)\Big)_X
\]\[
=\left< \DAn(\Hw{Z_h}-\unmu),\Tnmu(\Hw{Z_h}-\unmu)\right>
\]\[
\le\big(\rho_T\normX{Z_h-\unmu}^2+\en\big)\normX{\Tnmu(\Hw{Z_h}-\unmu)}.
\]

Then, as $Z_h\in B_X\left(\unmu,\alpha\right)$, we have 
\begin{equation}\label{seggrado}
\normX{\Hw{Z_h}-\unmu}<\dfrac{\rho_T}{\bmu}\alpha^2+\dfrac{\en}{\bmu}.
\end{equation}

In order to ensure that $H$ maps $B_X(\unmu,\alpha)$ into a part of itself, we are seeking the values of $\alpha$ such that $
\dfrac{\rho_T}{\bmu}\alpha^2+\dfrac{\en}{\bmu}\le\alpha$. This holds if $\alpha$ is between the two roots of the second order equation
$\romu\alpha^2-\bmu\alpha+\en=0,$ which are

\begin{equation}\label{roots}
\alpha_{\pm}=\frac{\bmu\pm\sqrt{\bmu^2-4\romu\en}}{2\romu}=\frac{\bmu}{2\romu}\left[1\pm\sqrt{1-\frac{4\romu\en}{\bmu^2}}\right].
\end{equation}

Observe that as $\taun\le1$, then $\alpha_{-}\le\alpha_{+}\le\dfrac{\beta_N(\mu)}{\rho_T}$. Consequently, if $\alpha_{-}\le\alpha\le\alpha_{+}$, there exists a unique solution $U_h(\mu)$ to (\ref{pb}) in the ball  $B_X(\unmu,\alpha)$.

To obtain (\ref{err}) observe that from (\ref{roots}), the lowest value (i.e. the best error bound) corresponds to $\alpha=\alpha_{-}=\dn$. To prove (\ref{efec}), let us define the error $E_h(\mu)=U_h(\mu)-\unmu$, and the residual $R(\mu)$, such that
\[
(R(\mu),V_h)_X=-\dual{\cR{\unmu}}=F(V_h;\mu)-A(\unmu,V_h;\mu)
\]\[
=A(U_h(\mu),V_;\mu)-A(\unmu,V_h;\mu).
\]

Note that, from (\ref{eps}), $\|R(\mu)\|_{X}=\en$. We observe that the following relation holds, for some $t^*\in(0,1)$:
\[
A(U_h(\mu),V_h;\mu)-A(\unmu,V_h;\mu)=\partial_1A\big(t^*U_h(\mu)+(1-t^*)\unmu,V_h;\mu\big)(E_h(\mu)).
\]
Thus, we have that
\[
\|R(\mu)\|_{X}^2=\Big[\partial_1A\big(t^*U_h(\mu)+(1-t^*)\unmu,R(\mu);\mu\big)-\partial_1A(\unmu,R(\mu);\mu)\Big](E_h(\mu))
\]\[
+\partial_1A(\unmu,R(\mu);\mu)(E_h(\mu)).
\]

Thus, thanks to Lemma \ref{LemmaRho}, and taking into account the definition of $\gamma_N(\mu)$ by (\ref{gammaN}), we obtain
\[
\normX{R(\mu)}^2\le\romu\normX{t^*(U_h(\mu)-U_N(\mu))}\normX{E_h(\mu)}\normX{R(\mu)}+\gamma_N(\mu)\normX{E_h(\mu)}\normX{R(\mu)}.
\]

Then $\en=\romu\normX{E(\mu)}^2+\|R(\mu)\|_{X}\le\gamma_{N}(\mu)\normX{E(\mu)}$. Since $0\le\taun\le1$  we have that
\[
\frac{2\romu}{\bmu}\dn\le\taun,
\]
and then $\dn\le2\en/\bmu$. It follows that
\[
\dn\le\frac{2\romu}{\bmu}\normX{E(\mu)}^2+\frac{2\gamma_{N}(\mu)}{\bmu}\normX{E(\mu)}.
\]

Thanks to (\ref{err}), we know that $\normX{E(\mu)}\le\dn$, then $\dfrac{2\rho_T}{\beta_N(\mu)}\normX{E(\mu)}\le\taun$. It follows (\ref{efec}), i.e.,
\[
\dn\le\left[\frac{2\gamma_{N}(\mu)}{\bmu}+\taun\right]\normX{U(\mu)-\unmu}.
\]
\end{proof}

\section{Approximation of the eddy viscosity term}\label{sec::EIM}
In this section we approximate the non-linear turbulent eddy viscosity term by the Empirical Interpolation Method \cite{EIM1,EIM2}. The EIM allows the construction an offline tensorized representation of this term that will be used in the online calculations.

Let us denote $g(\mu):=g(x;\wk_h(\mu))=|\nabla\wk_h(\mu)|(x)$. The finality of using the EIM is decoupling the $\mu$-dependence from the spatial dependence of the function $g(\mu)$, i.e., 
\begin{equation}
g(\mu)\approx\mathcal{I}_{M}[g(\mu)],
\end{equation}
where we denote by $\mathcal{I}_{M}[g(\mu)]$ the empirical interpolate of $g(\mu)$. The EIM, consists of constructing a reduced-basis space $W_M=\mbox{span}\{q_1(\mu),\dots,q_M(\mu)\}$, selecting these  basis functions by a greedy procedure, with snapshots of $g(\mu)$. With this technique, we are able to approximate the non-linear Smagorinsky term by a trilinear form, in the following way: $a_S(\wk_N;\wk_N,\vk_N;\mu)\approx \hat{a}_{S}(\wk_N,\vk_N;\mu)$, where, 
\[
\hat{a}_{S}(\wk_N,\vk_N;\mu)=\D\sum_{k=1}^M\sigma_k(\mu)s(q_k,\wk_h,\vk_h),
\]
with
$
s(q_k,\wk_h,\vk_h)=\D\intK{(C_Sh_K)^2q_k\nabla\wk:\nabla\vk}.
$

Here $\sigma_k(\mu),$ for $k=1,\dots,M$, is the solution of a lower-triangular linear system, where the second member is the value of $g(x;\wk_h(\mu))$ in some certain points $x_i$. We refer to \cite{EIM1} for more details. 
%

This technique allows us to linearize the eddy viscosity term. Let us recall the RB problem, with this last approximation of the Smagorinsky term:
\begin{equation}\label{RBPB}\left\{\begin{array}{l}
\mbox{Find } (\uk_N,p_N)\in Y_N\times M_N\mbox{ such that }\forall\vk_N\in Y_N, \forall q_N\in M_N\vspace{0.2cm}\\
a(\uk_N,\vk_N;\mu)+b(\vk_N,p_N;\mu)+ \hat{a}_S(\wk_N,\vk_N;\mu) \\
+c(\uk_D,\uk_N,\vk_N;\mu)+c(\uk_N,\uk_D,\vk_N;\mu)+c(\uk_N,\uk_N,\vk_N;\mu)=F(\vk_N;\mu)\\
b(\uk_N,q_N;\mu)=0\end{array}\right.
\end{equation}

The solution $(\uk_N(\mu),p_N(\mu))\in X_N$ of (\ref{RBPB}) can be expressed as a linear combination of the basis functions:
\[
\uk_N(\mu)=\sum_{j=1}^{2N}u_j^N(\mu)\zeta^{\vk}_j,\quad p_N(\mu)=\sum_{j=1}^Np_j^N(\mu)\xi^p_j.
\]

Taking into account this representation, for the bilinear terms in (\ref{RBPB}), we store the parameter-independent matrices for the offline phase, as in \cite{Manzoni}, defined as:
\begin{equation}\label{RB matrices}
\begin{array}{ll}
(\mathbb{A}_N)_{ij}=a(\zeta^{\vk}_j,\zeta^{\vk}_i),\; (\mathbb{B}_N)_{li}=b(\zeta^{\vk}_i,\xi^p_l),
&\; i,j=1,\dots,2N, l=1,\dots,N.\vspace{0.1cm}\\
(\mathbb{D}_N)_{ij}=c(\uk_D,\zeta^{\vk}_j,\zeta^{\vk}_i)+c(\zeta^{\vk}_j,\uk_D,\zeta^{\vk}_i),
&\; i,j=1,\dots,2N.
\end{array}
\end{equation}

For the convective and the Smagorinsky terms, we need to store a parameter-independent tensors of order three for the offline phase,  defined as:
\begin{equation}\label{RB Tensors}
\begin{array}{ll}
(\mathbb{C}_N(\zeta^{\vk}_s))_{ij}=c(\zeta^{\vk}_s,\zeta^{\vk}_j,\zeta^{\vk}_i), 
&\; i,j,s=1,\dots, 2N, \vspace{0.1cm}\\
(\mathbb{S}_N(q_s))_{ij}=s(q_s,\zeta^{\vk}_j,\zeta^{\vk}_i),
&\;i,j=1,\dots,2N, s=1,\dots, M.
\end{array}
\end{equation} 

With this tensor representation, it holds that 
\[
c(\uk_N,\zeta^{\vk}_j,\zeta^{\vk}_i)=\sum_{s=1}^{2N}u^N_s(\mu)\mathbb{C}_N(\zeta^{\vk}_s) \quad \mbox{and} \quad
\hat{a}_S(\zeta^{\vk}_j;\zeta^{\vk}_i)=\sum_{s=1}^{2N}\sigma_s(\mu)\mathbb{S}_N(q_s),
\]
and thanks to that, we are able to solve problem (\ref{RBPB}), linearized by a semi-implicit evolution approach. We remark that the treatment of the approximation of the eddy viscosity term in the offline/online phase is similar to the treatment of the convective term, thanks to the tensorization done in this section.

\section{Numerical results}\label{sec::results}

In this section, we present some numerical tests for the reduced order Smagorinsky model. We will consider two test cases: the Backward-facing step problem and the Lid-driven cavity problem. For both cases, we consider a range of Reynolds number for which it is known that a steady regime takes place. We obtain rates of speed-up of the computational time from several hundreds to several thousands.

\subsection{Backward-facing step flow (2D)} \label{sec::BFS} 

In this test, we show the numerical results of a Smagorinsky reduced order model for the backward-facing step (\textit{cf.} \cite{armali}).  
%
The backward-facing step flow is laminar and reaches a steady state solution roughly up to $Re\simeq1000$, then it becomes transitional up to $Re\simeq5000$. For larger values, the regime become turbulent. (\textit{cf.} \cite{TomasSmago}). The parameter that we are considering for this numerical test case is the Reynolds number, with values $\mu=Re\in\cD=[50,450]$. This means that the regime we consider in this test is fully laminar.

For the offline phase, we compute the FE approximation with the Taylor-Hood finite element, i.e, we consider $\mathbb{P}2-\mathbb{P}1$ for velocity-pressure. The mesh selected for this problem is composed of 10842 triangles and 5703 nodes. The FE steady state solution is computed through a semi-implicit evolution approach, and we conclude that the steady solution is reached when the relative error between two iterations is below $\varepsilon_{FE}=10^{-10}$. The numerical scheme to solve the Smagorinsky model in each time step reads

\begin{equation}
\left\{\begin{array}{l}
\mbox{Find } (\uk_h^{n+1},p_h^{n+1})\in Y_h\times M_h\mbox{ such that }\forall\vk_h\in Y_h, \forall q_h\in M_h\vspace{0.2cm}\\
\left(\dfrac{\uk_h^{n+1}-\uk_h^n}{\Delta t}\right)_{\Omega}+a(\uk_h^{n+1},\vk_h;\mu)+b(\vk_h,p_h^{n+1};\mu)+a_S(\wk_h^{n};\wk_h^{n+1},\vk_h;\mu)\vspace{0.1cm}\\
c(\uk_h^{n},\uk_h^{n+1},\vk;\mu)+c(\uk_D,\uk_h^{n+1},\vk;\mu)+c(\uk_h^{n+1},\uk_D,\vk;\mu)=F(\vk_h;\mu)\vspace{0.2cm}\\
b(\uk_h^{n+1},q_h;\mu)=0
\end{array}\right.
\end{equation}

To implement the Greedy algorithm, we compute beforehand the inf-sup constant $\beta_N(\mu)$, (\ref{betaN}), and the Sobolev embedding constant $C_T$, as both appear in the \textit{a posteriori} error bound estimator. Due to the fact that $U_N(\mu)$ is intended to be a good approximation of $U_h(\mu)$, in practice we use the value of $\beta_h(\mu)$ in place of $\beta_N(\mu)$. To compute the inf-sup stability factor, we consider the heuristic strategy introduced in \cite{infsup}. This heuristic technique consists in interpolating the $\beta_h(\mu)$ map by the adaptive Radial Basis Function algorithm (c.f. \cite{rbf_adap}), for some selected values of $\mu\in\cD$. On the other hand, to compute the Sobolev embedding constant, we use the fixed point algorithm described in \cite{Manzoni} (section 8). 

In Fig. \ref{fig::BetaStep} (left) we compare $\rho_{\mub}/\beta_{\mub}(\mu)$, described in \cite{Deparis}, and $\rho_T/\beta_h(\mu)$. These quantities are crucial for the number of basis functions necessary to assure that $\taun<1$ . Since for our problem, $\rho_T/\beta_h(\mu)<\rho_{\mub}/\beta_{\mub}(\mu)$, the number of bases needed to guarantee $\taun<1$ is lower when we use the norm $\normT{\cdot}$ instead the natural norm.

\begin{figure}[h]
\centering
\includegraphics[width=0.45\linewidth]{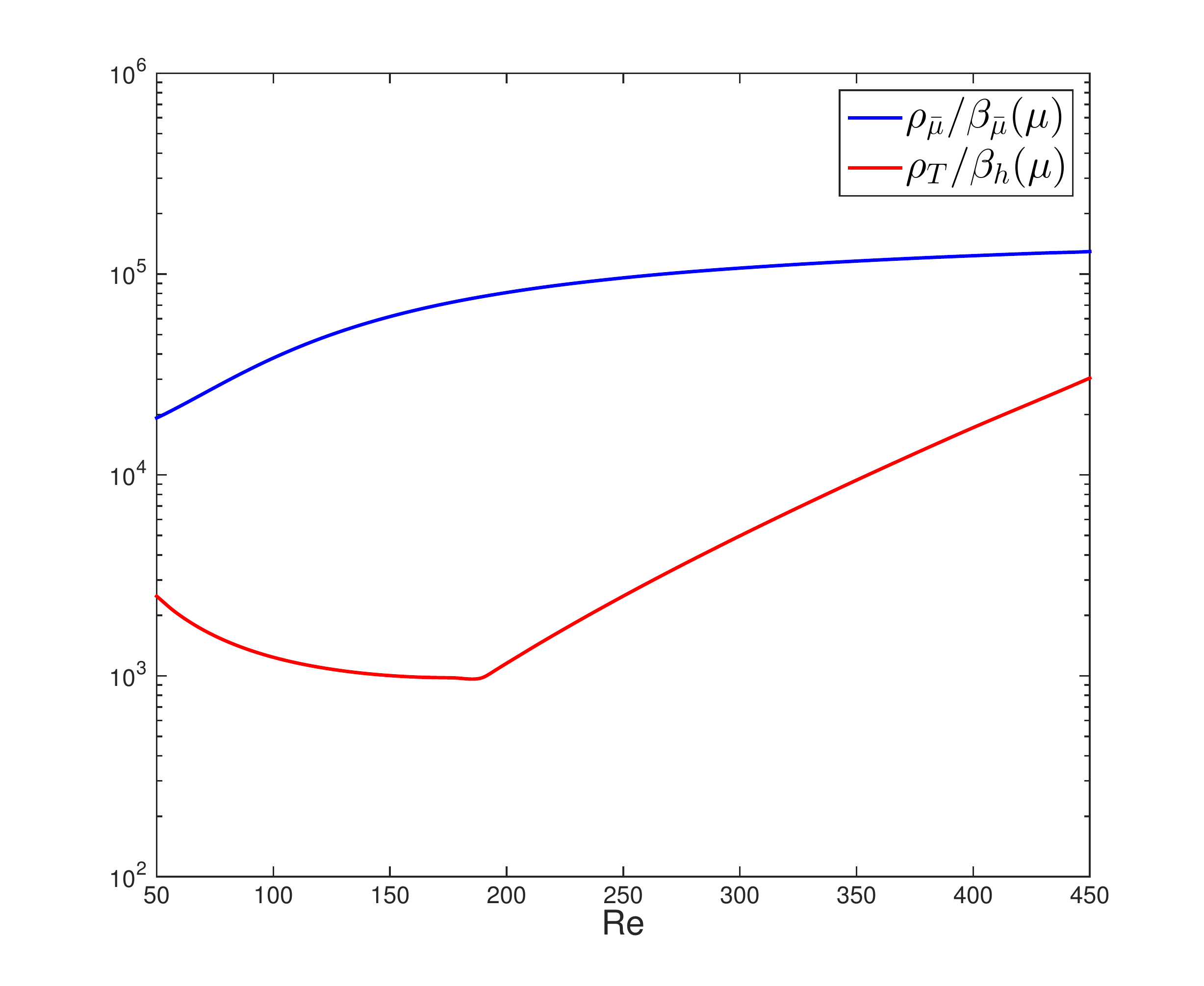}
\includegraphics[width=0.45\linewidth]{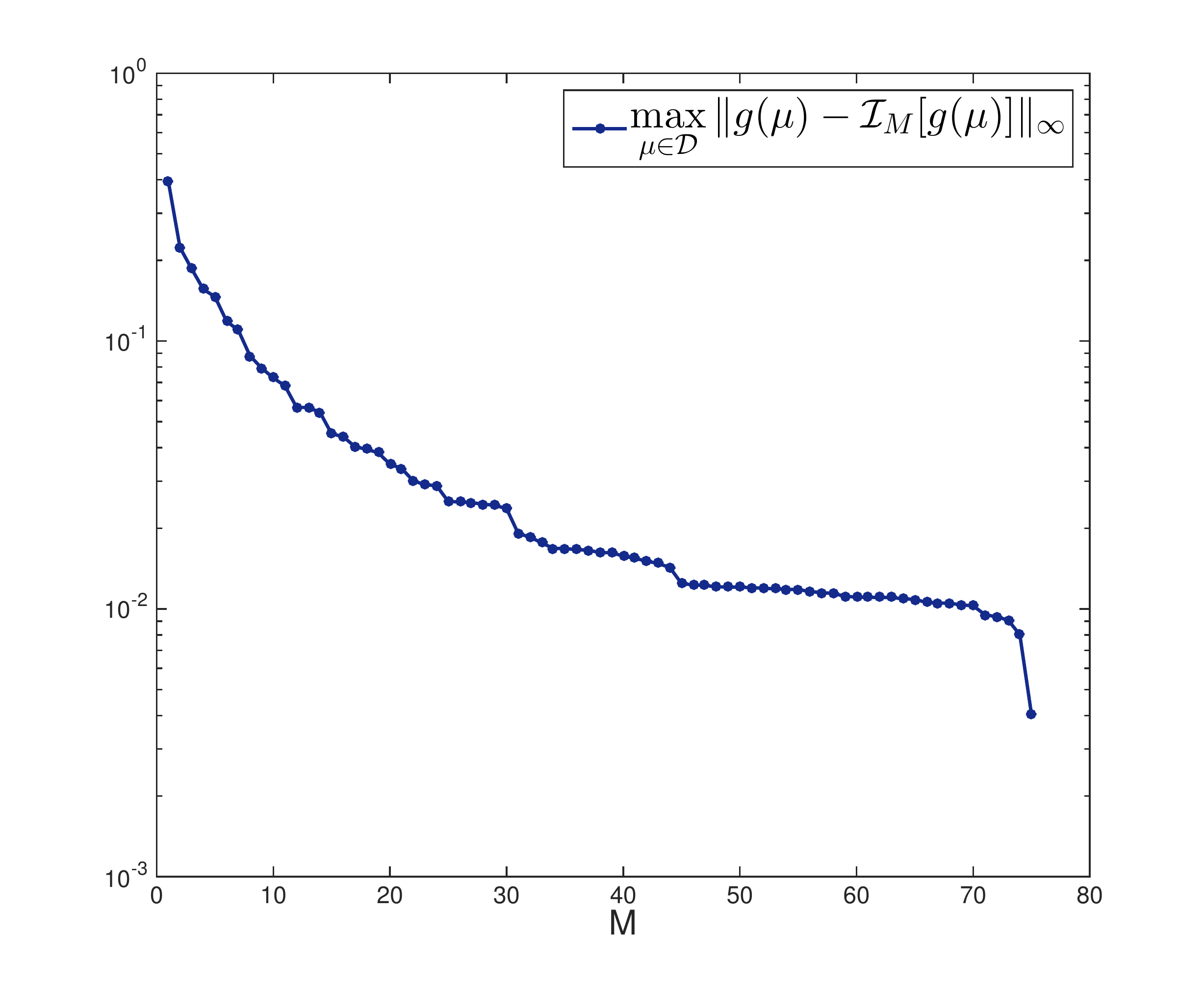}
\vspace{-0.3cm}
\caption{Left: comparison $\rho_{\mub}/\beta_{\mub}(\mu)$ and $\rho_T/\beta_h(\mu)$. Right: Convergence of the EIM algorithm.}\label{fig::BetaStep}
\end{figure} 

To compute our reduced-basis space, we start by computing the basis functions to construct the EIM to approximate the eddy diffusion term. To evaluate the error between $g(\mu)$ and the empirical interpolation, we precompute a certain number of snapshots $\uk_h(\mu)$, $\mu\in\cD_{EIM}$. We stop the construction of the EIM bases when we reach a relative error below $\veps_{EIM}=5\cdot10^{-3}$. This error is reached for 75 basis functions for the EIM. In Fig. \ref{fig::BetaStep} (right), we show the evolution of the relative error, in the infinite norm, between $g(\mu)$ and its empirical interpolate.

Also, to obtain an initial guess such that $\tau_N\lesssim1$, as needed by our error estimator, we first approximate the reduced manifold with some POD modes (see e.g. \cite{Ballarin}), and then we start our Greedy algorithm. We select 10 POD modes using the snapshots computed for the EIM to start the Greedy algorithm.

In Fig. \ref{fig::Greedyerr} (left), we can observe the evolution of the \textit{a posteriori} error bound within the Greedy algorithm.  We observe that it is indeed a good error estimator, with an efficiency factor close to 10 for all $\mu\in\cD$ (Fig. \ref{fig::Greedyerr} (right)). Due to Theorem \ref{Teorprinc}, $\dn$ exists when $\tau_N(\mu)\le1$. While $\tau_N(\mu)>1$, we use as \textit{a posteriori} error bound estimator the proper $\tau_N(\mu)$. We stop the Greedy algorithm when we reach a tolerance of $\veps_{RB}=7\cdot10^{-5}$, obtained when $N=N_{max}=17$. In Fig. \ref{fig::Greedyerr} (right) we show the value of the \textit{a posteriori} error bound estimator and the relative error for all $\mu\in\cD$, at $N=N_{\max}$.

\begin{figure}[h]
\centering
\includegraphics[width=0.45\linewidth]{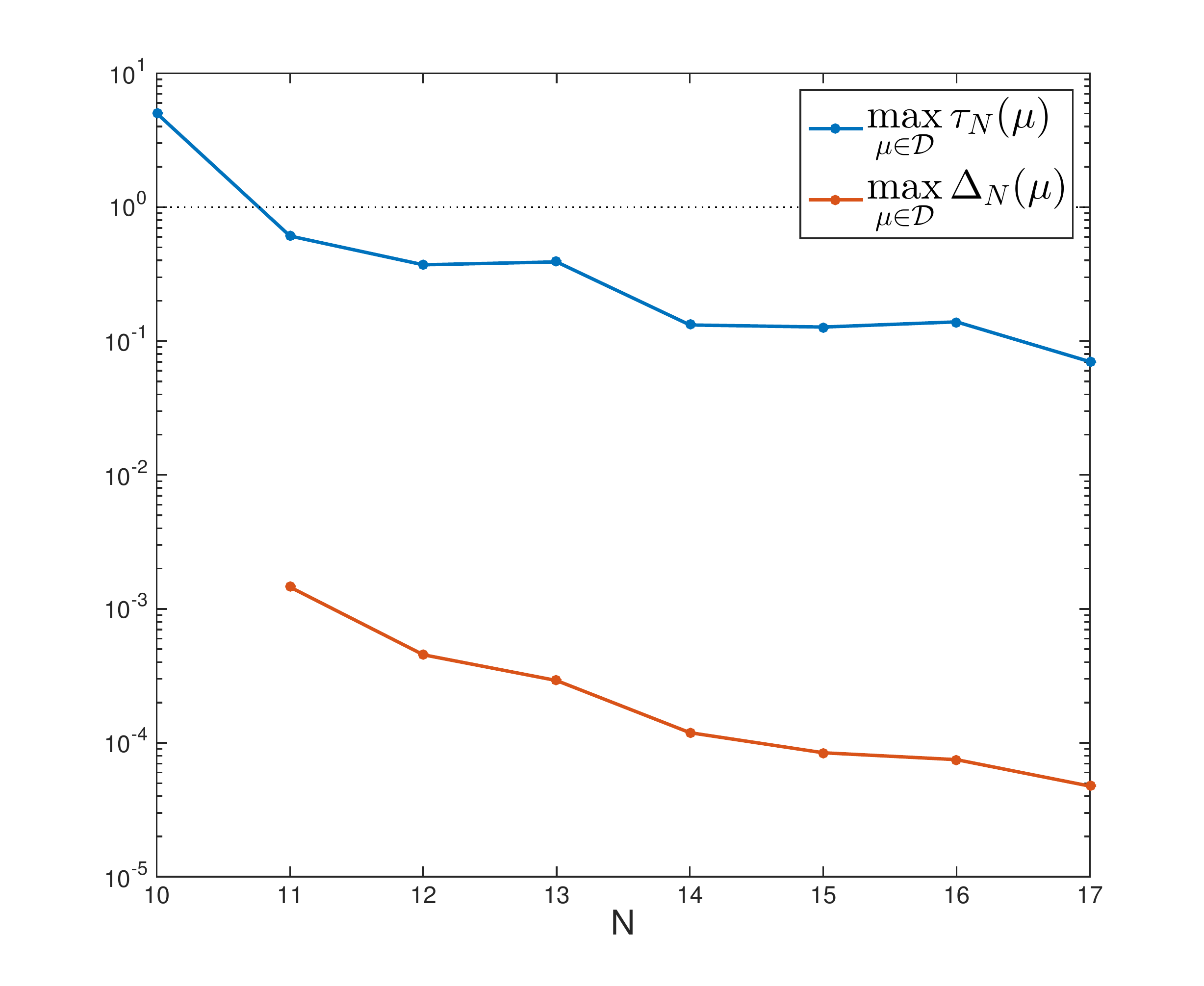}
\includegraphics[width=0.45\linewidth]{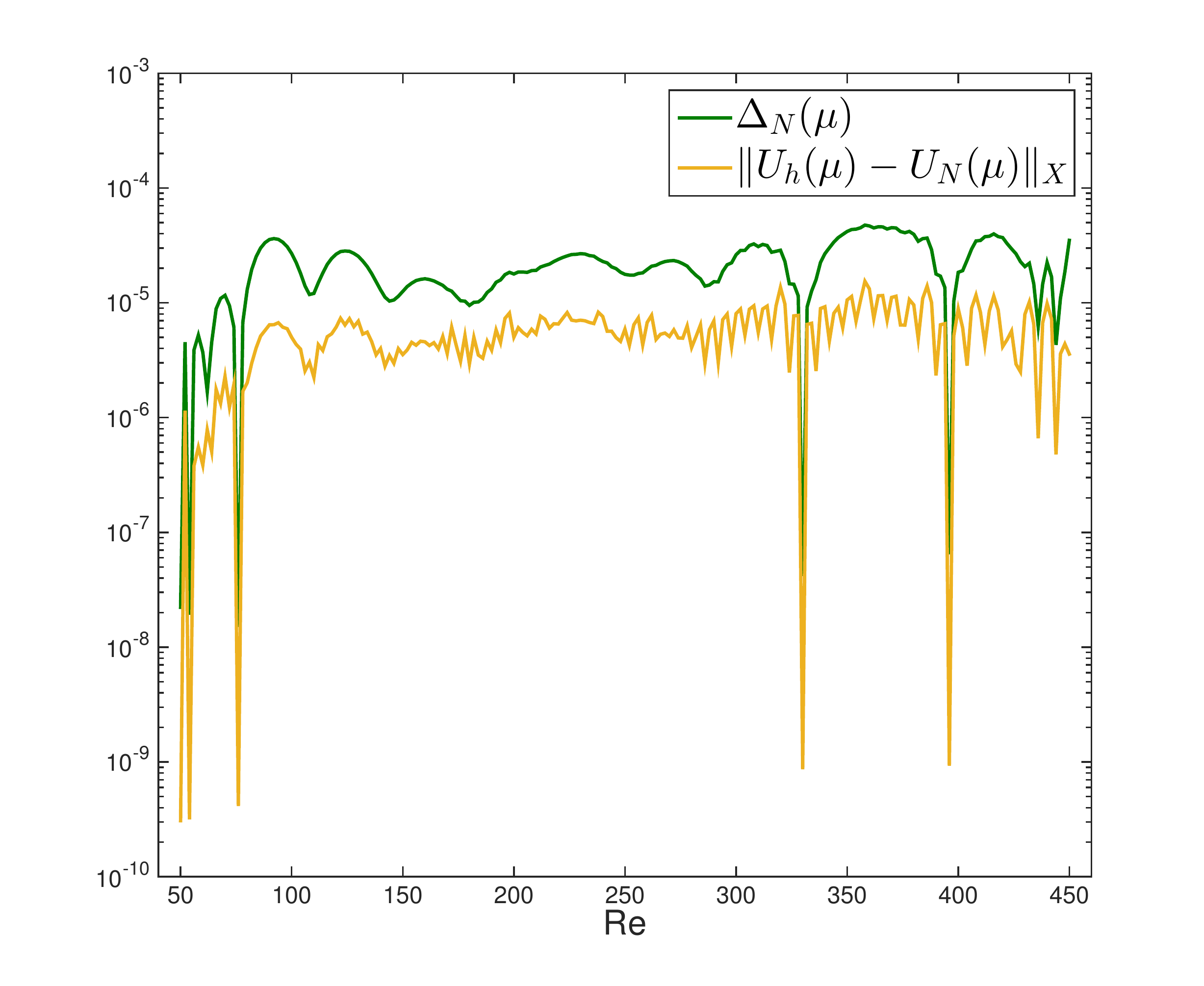}
\vspace{-0.3cm}
\caption{Left: Convergence of the Greedy algorithm. Right: Value of $\Delta_{N_{\max}(\mu)}$ and the error between the FE solution and the RB solution.}\label{fig::Greedyerr}
\end{figure}

To compute the error in the EIM approximation of the Smagorinsky eddy-viscosity term, we define the following errors:

\begin{equation}
e_S(\mu)=\D\sup_{\vk \in Y} |a_S(\wk_h;\wk_h,\vk;\mu)-a_S(\wk_N;\wk_N,\vk;\mu)|,\;
n_S(\mu)=\D\sup_{\vk \in Y} a_S(\wk_h;\wk_h,\vk;\mu)
\end{equation}

In Fig. \ref{fig::SmagoErr} we show the value of $e_S/n_S$, the normalized error of the Smagorinsky term EIM approximation, for all $\mu$ in $\cD$.

\begin{figure}[h]
\centering
\includegraphics[width=0.6\linewidth]{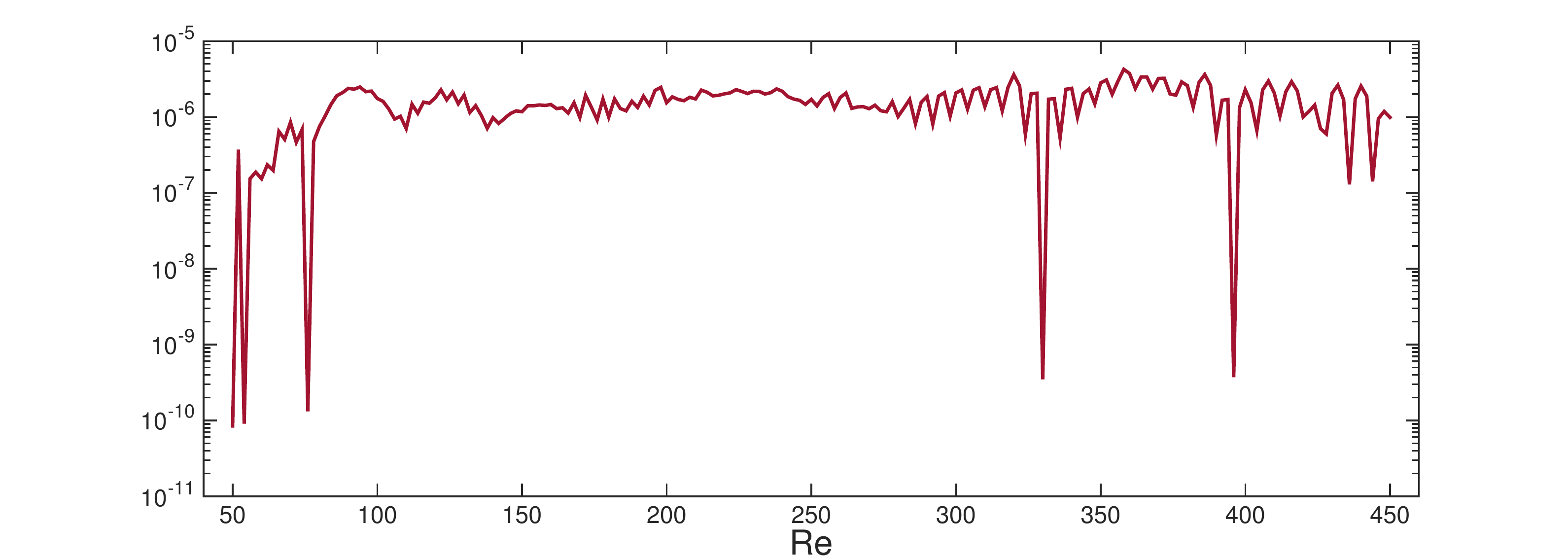}\vspace{-0.3cm}
\caption{Normalized error of the EIM Smagorinsky term approximation.}\label{fig::SmagoErr}
\end{figure}

We can observe that the good approximation of the RB solution provides a good approximation for the Smagorinsky term, since the error between FE and RB solution in Fig. \ref{fig::Greedyerr} (right) and the error in Fig. \ref{fig::SmagoErr} are similar.

In Fig.\ref{fig::SolEscalon} we show a comparison between the FE velocity solution (left) and the RB velocity solution (right) for a chosen parameter value $\mu=320$. Note that both images are practically equal, as the error between both solutions is of order $10^{-6}$.

\begin{figure}[h]
\includegraphics[width=\linewidth]{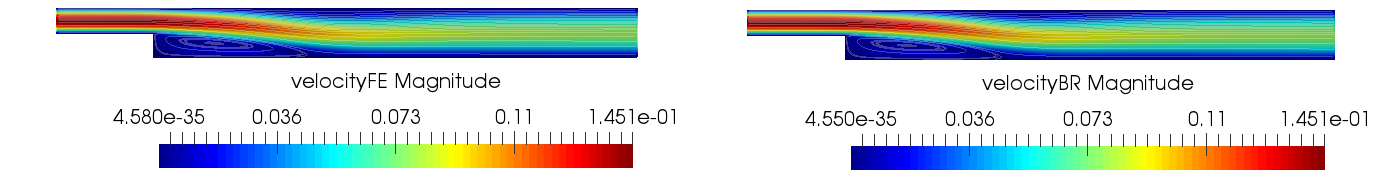}
\caption{FE solution (left) and RB solution (fight) for $\mu=320$.}\label{fig::SolEscalon}
\end{figure}

In Table \ref{tab:: resBFS}, we show the results obtained for several values of $\mu$ in $\cD$, in particular, we compare the computational time for the computation of the FE solution and the RB solution in the online phase. We also show the speedup  rate in the computation of the RB solution, and the relative errors in $H^1$-norm for velocity and in $L^2$ for pressure. We observe a dramatic reduction of the computational time, with speed-up rates over 1000 for large Reynolds number, with relative errors below the Greedy tolerance. The offline phase took 3 days and 10 hours to be completed. In this time, we are considering the time to construct the RBF functions for the stability factor $\beta_h$, the computation of the snapshots necessaries for the EIM, and the Greedy algorithm.

\begin{table}[h]
\[\hspace{-0.1cm}
\begin{tabular}{l|ccccc}
\hline
Data &$\mu=56$&$\mu=132$&$\mu=236$&$\mu=320$&$\mu=450$\\
\hline
$T_{FE}$&237.88s & 503.57s & 1055.91s&1737.74s&2948.11s\\ 
$T_{online}$& 1.40s& 1.51s& 1.69s&1.72s&2.02s\\
\hline
speedup& 169 & 333& 622 &1008&1458\\
\hline
$\|\uk_h-\uk_N\|_T$&$3.77\cdot10^{-7}$&$5.33\cdot10^{-6}$ & $6.58\cdot10^{-6}$ &$1.36\cdot10^{-5}$&$3.57\cdot10^{-6}$\\
$\|p_h-p_N\|_0$&$1.94\cdot10^{-8}$&$6.97\cdot10^{-8}$ & $2.1\cdot10^{-7}$  &$4.82\cdot10^{-7}$&$9.02\cdot10^{-8}$\\
\hline
\end{tabular}\]\vspace{-0.3cm}\caption{Computational time for FE solution and RB online phase, with the speedup and the relative error.}\label{tab:: resBFS}
\end{table}

\subsection{Lid-driven Cavity flow (2D)}

In this test, we apply the reduced-order Smagorinsky turbulence model to the Lid-driven cavity problem.  We consider the non homogeneous Dirichlet boundary condition given by $g(x)=1$ on the lid boundary $\Gamma_{D_g}$, with homogeneous Dirichlet condition on $\Gamma_{D_0}$.
  
For this test, we also consider the Reynolds number as a parameter, ranging in $\cD=[1000,5100]$. The 2D lid-driven cavity flow has a steady solution up to Reynolds 7500 (c.f. \cite{TomasSmago}); thus in this range, a steady solution is well known to exist. We consider the same finite elements as the previous problem (Taylor-Hood), and we use a regular mesh with 5000 triangles and 2601 nodes. 

In Fig. \ref{fig::betaCav} (left), we show the comparison between $\rho_{\mub}/\beta_{\mub}(\mu)$ and $\rho_T/\beta_h(\mu)$  as in Section \ref{sec::BFS}. Again, the number of basis functions needed to guarantee that $\tau_N(\mu)<1$ in this problem is lower if we chose the norm $\normT{\cdot}$ instead of choosing the natural norm.  
\begin{figure}[h]
\centering
\includegraphics[width=0.45\linewidth]{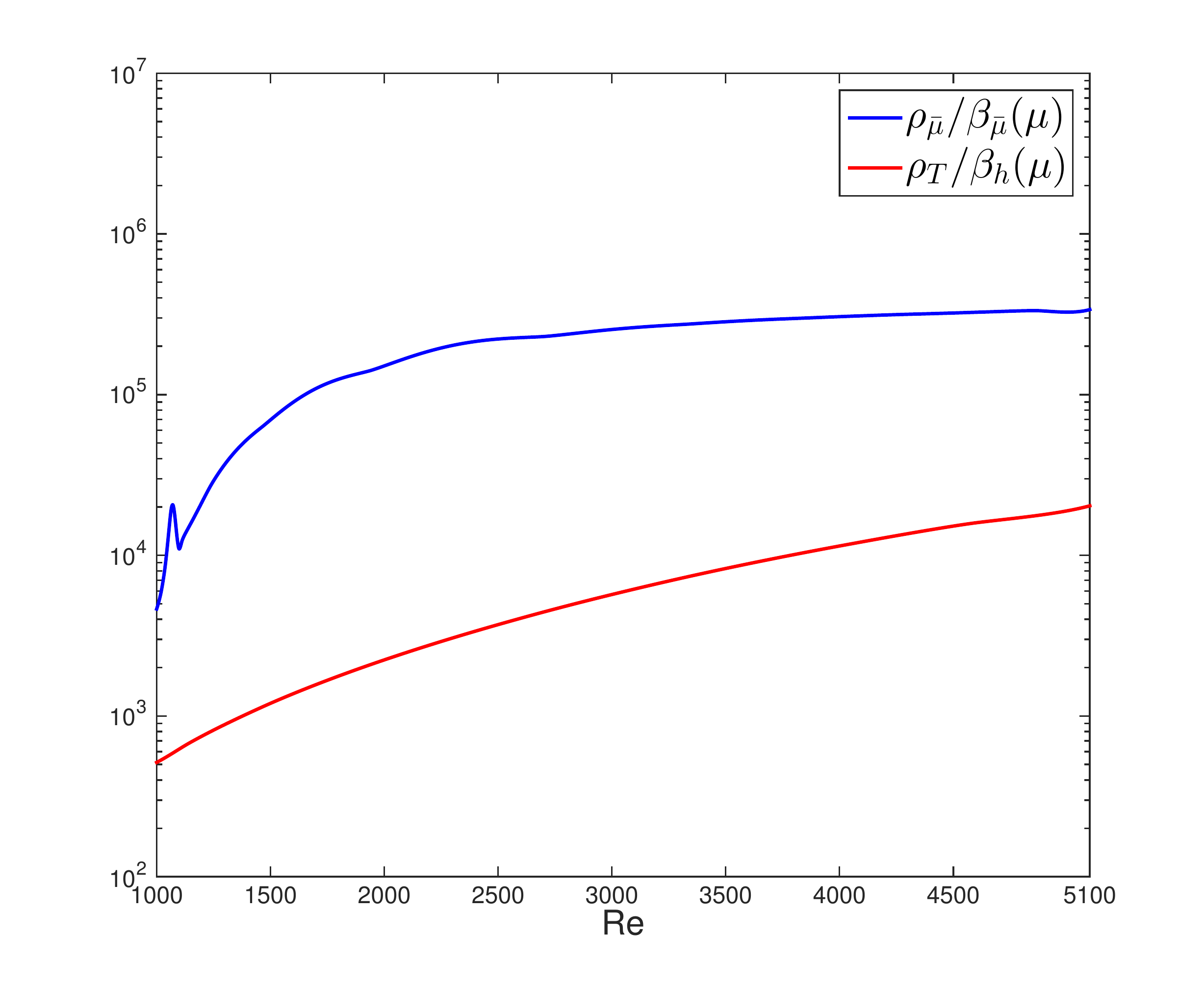}
\includegraphics[width=0.45\linewidth]{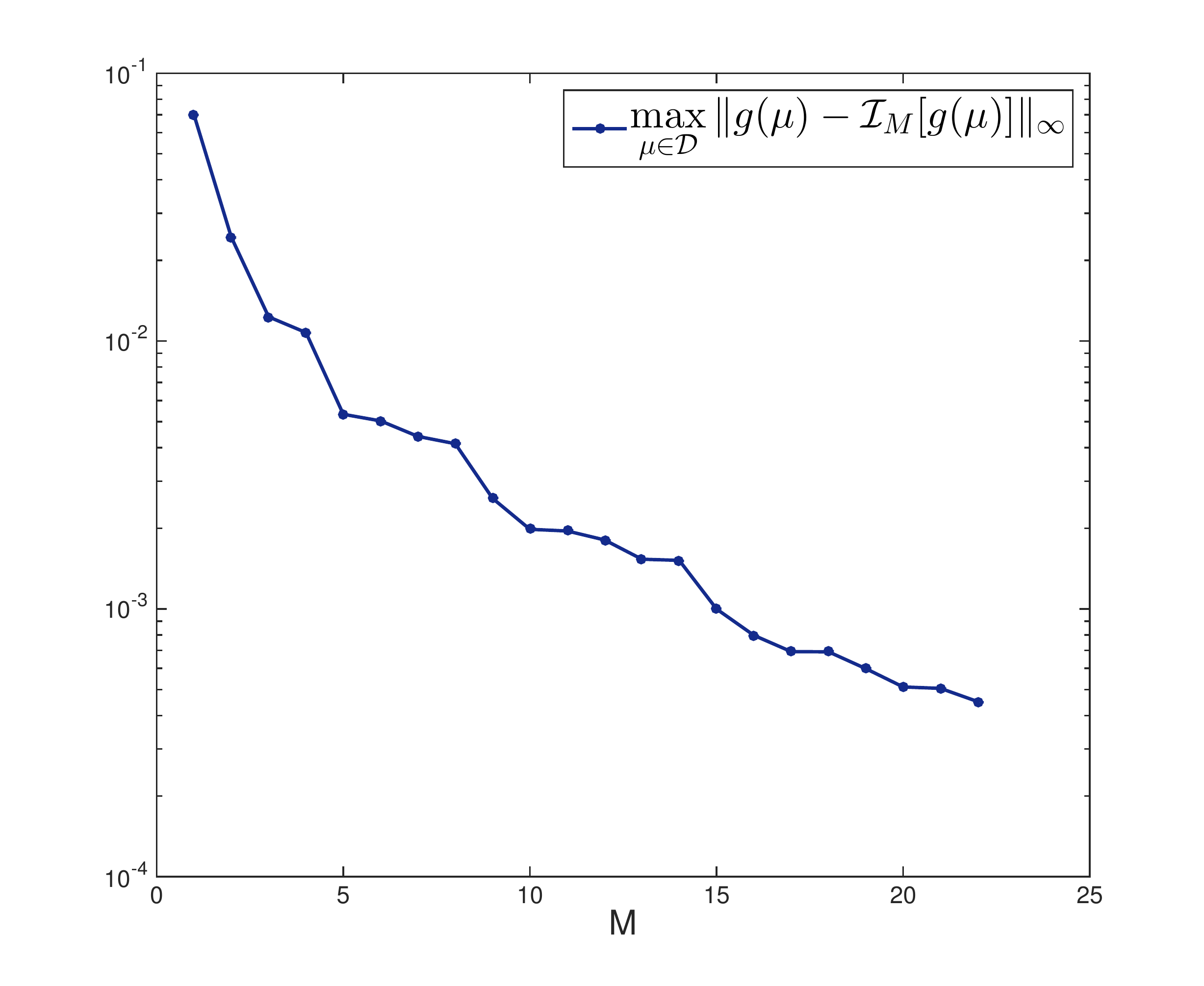}
\vspace{-0.3cm}
\caption{Left: comparison $\rho_{\mub}/\beta_{\mub}(\mu)$ and $\rho_T/\beta_h(\mu)$. Right: Convergence of the EIM algorithm.}\label{fig::betaCav}
\end{figure}
In this numerical test, we need $M_{\max}=22$ basis functions in the EIM algorithm until reaching the tolerance for the relative error of $\varepsilon_{EIM}=5\cdot10^{-4}$. In Fig. \ref{fig::betaCav} (right), we show the convergence of the EIM algorithm.

For the Greedy algorithm, we prescribe a tolerance of $\varepsilon_{RB}=5\cdot10^{-5}$. This tolerance is reached for $N_{\max}=12$ basis functions. Note that, in this case, $N=8$ basis functions are needed in order to assure that $\tau_N(\mu)<1$ for all $\mu$ in $\cD$. In Fig. \ref{fig::GreedyCav} (left) we show the convergence of the Greedy algorithm, and in Fig. \ref{fig::GreedyCav} (right) we show the value of the error and the a posteriori error bound for all $\mu$ in $\cD$.
\begin{figure}[h]
\centering
\includegraphics[width=0.45\linewidth]{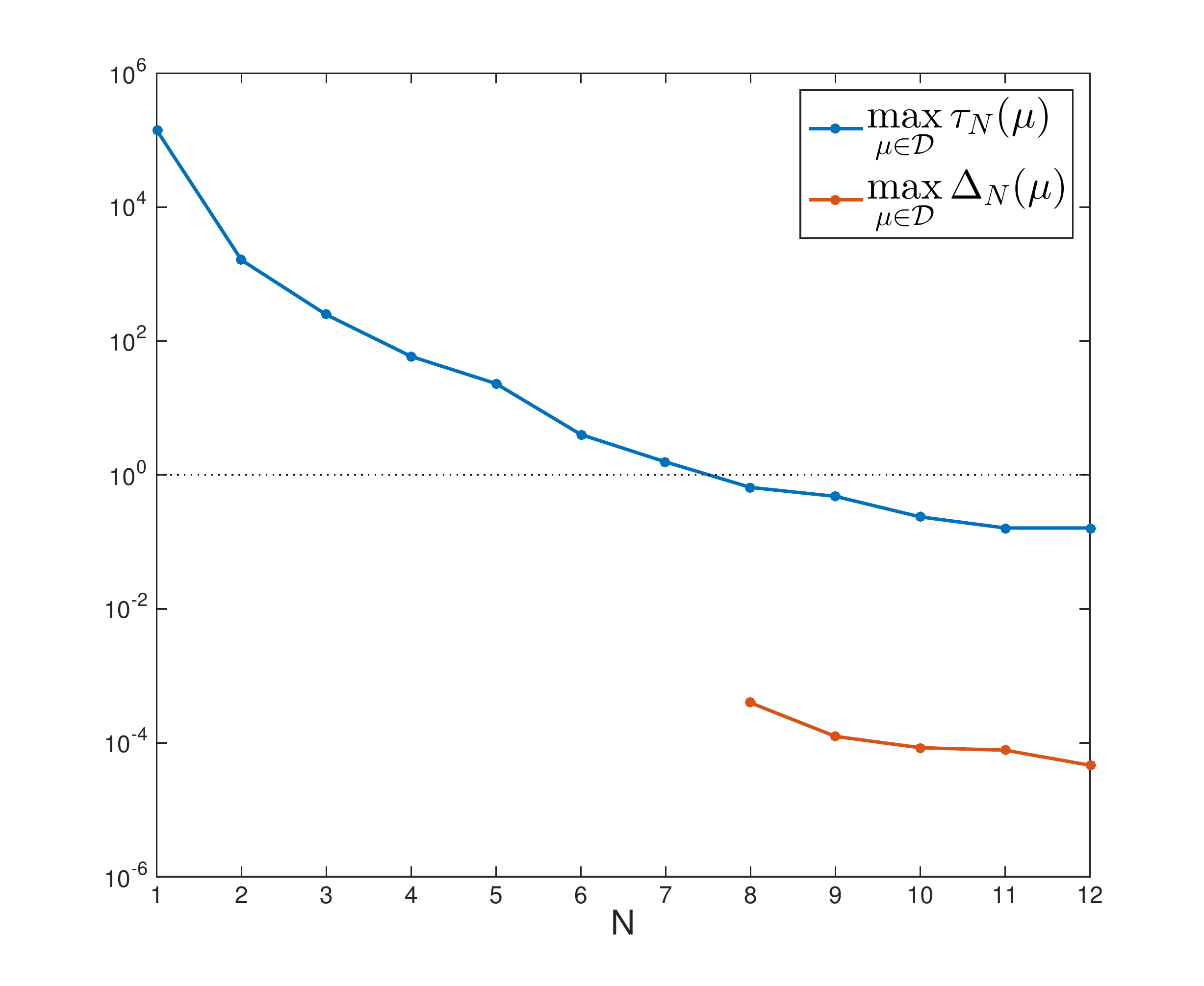}
\includegraphics[width=0.45\linewidth]{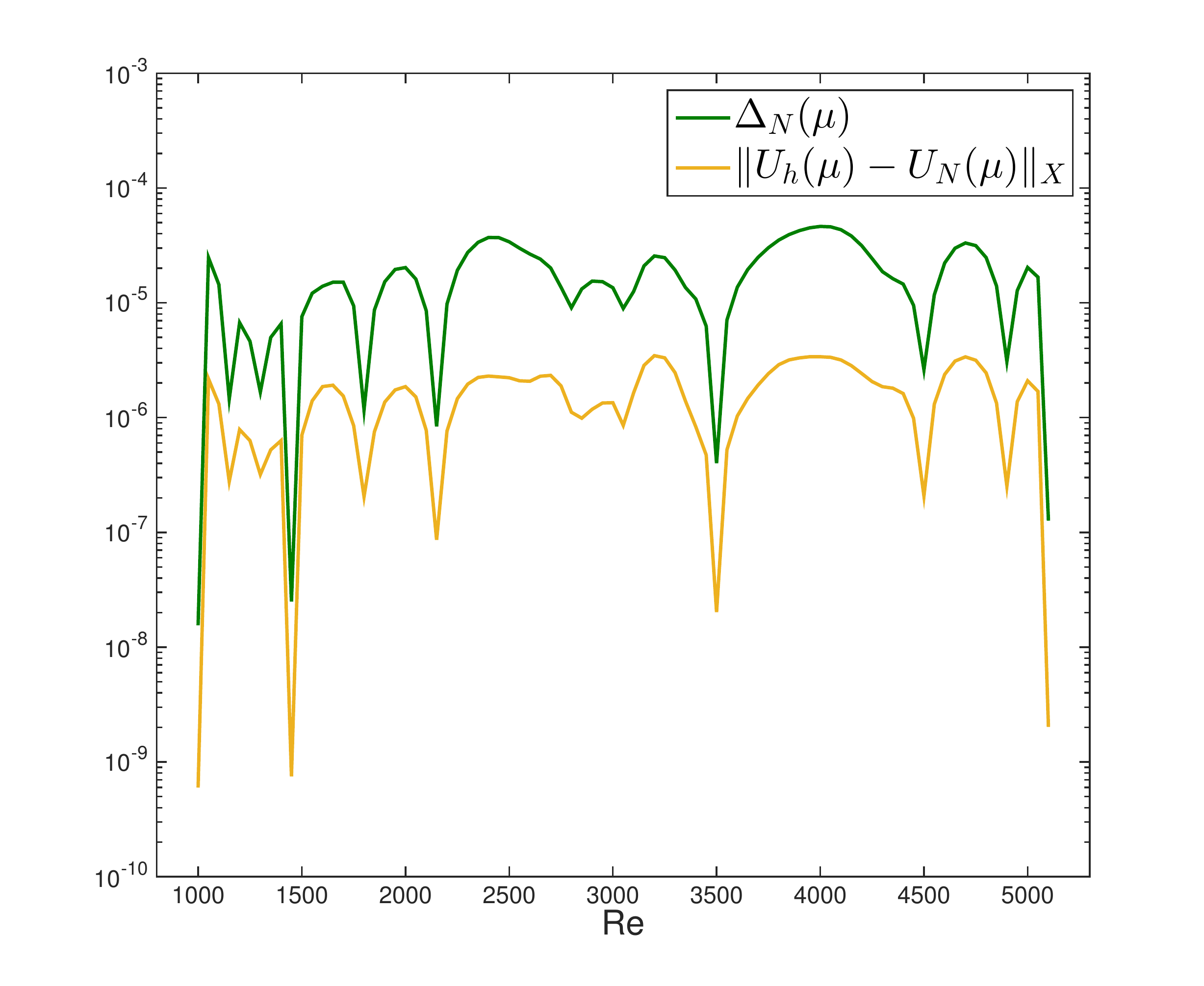}
\vspace{-0.3cm}
\caption{Left: Convergence of the Greedy algorithm. Right: Value of $\Delta_{N_{\max}}(\mu)$ and the error between the FE solution and the RB solution.}\label{fig::GreedyCav}
\end{figure}

As in Section \ref{sec::BFS}, we compute the error $e_S(\mu)/n_S(\mu)$, in order to compute the error in the EIM approximation of the Smagorinsky term. In Fig. \ref{fig::ESmagoC}, we show this error for this numerical test. Again, a good approximation between FE and RB solution, provides a good EIM approximation of the Smagorinsky term.
\begin{figure}[H]
\centering
\includegraphics[width=0.6\linewidth]{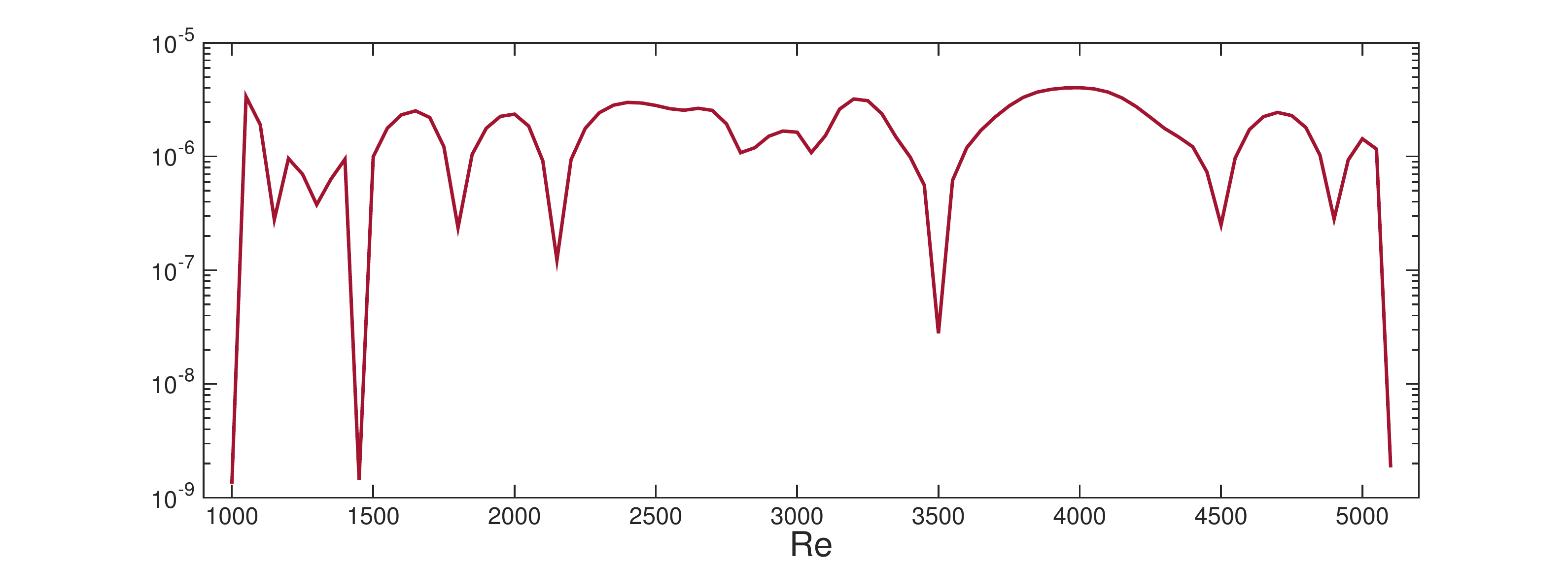}
\caption{Normalized error of the EIM Smagorinsky term approximation.}\label{fig::ESmagoC}
\end{figure}
In Fig. \ref{fig::CavSol} we show a comparison between the FE velocity solution and the RB velocity solution for a chosen parameter value $\mu=4521$. Again, both images are practically equal, as the error between both solutions is of order $10^{-7}$.

\begin{figure}[h]
\centering
\includegraphics[scale=0.25]{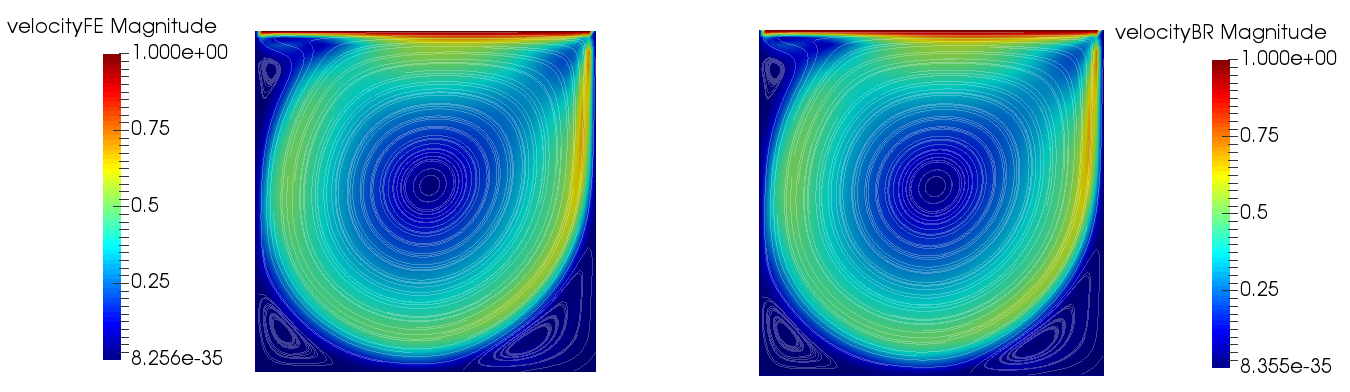}
\vspace{-0.3cm}
\caption{FE solution and RB solution for $\mu=4521$.}\label{fig::CavSol}
\end{figure}

Finally, we show in Table \ref{tab:: resCav} a summary of the results obtained for several values of $\mu$ in $\cD$. For this test, we also observe a dramatic speedup in the computation of the numerical solution, even larger than in the Backward-facing step test. These large speed-up factors are possibly due to the high turbulent levels of viscosity introduced by the Smagorinsky turbulence model. The offline phase of this test took 2 days to be completed.

\begin{table}[h]
\[\hspace{-0.1cm}
\begin{tabular}{l|ccccc}
\hline
Data &$\mu=1610$&$\mu=2751$&$\mu=3886$&$\mu=4521$&$\mu=5100$\\
\hline
$T_{FE}$&638.02s & 1027.62s & 1369.49s&1583.08s&1699.52s\\
$T_{online}$& 0.47s& 0.47s& 0.47s&0.49s&0.52s\\
\hline 
speedup& 1349 & 2182 & 2899 & 3243 & 3227\\
\hline
$\|\uk_h-\uk_N\|_T$&$1.91\cdot10^{-6}$&$1.87\cdot10^{-6}$ & $3.28\cdot10^{-6}$ &$6.26\cdot10^{-7}$&$3.17\cdot10^{-9}$\\
$\|p_h-p_N\|_0$&$1.18\cdot10^{-7}$&$3.65\cdot10^{-7}$ & $3.78\cdot10^{-7}$  &$8.34\cdot10^{-8}$&$1.88\cdot10^{-9}$\\
\hline
\end{tabular}\]
\vspace{-0.3cm}\caption{Computational time for FE solution and RB online phase, with the speedup and the relative error.}\label{tab:: resCav}
\end{table}
\section{Conclusions}
In this paper we have developed a reduced basis Smagorinsky model, using the EIM to linearize the non-linear eddy viscosity of the Smagorinsky model. 

We have developed an \textit{a posteriori } error bound estimator for the Smagorinsky model, extending the theory in the literature for the incompressible Navier-Stokes equations (\textit{e.g} \cite{Manzoni},\cite{Deparis}). With the \textit{a posteriori } error bound estimator we can compute the offline phase in a efficient way, including the computation of the inf-sup stability interpolator, that provides a fast reliable approximation of the inf-sup stability factor value for each parameter value.

We have presented numerical results for two benchmark cases, where we have shown the accuracy of our reduced model and the dramatic reduction of the computational time for both cases, which is typically divided by several thousands. This high speed-up rate is possibly due to the high dissipative effect of the Smagorinsky turbulence model. Extensions to less dissipative turbulence models of VMS kind are in progress.

\section*{Acknowledgments}
We acknowledge Professor Yvon Maday (LJLL) for his support and guidelines during this research. This work has been supported by Spanish Government Project MTM2015-64577-C2-1-R and COST Action TD1307. 

\bibliographystyle{plain}
\bibliography{RefEscalon.bib}
\end{document}